\documentclass[12pt,twoside]{amsart}
\usepackage{amssymb}
\usepackage{xypic}
\usepackage{color}

\title{Minimal model theory for log surfaces} 
\author{Osamu Fujino} 
\subjclass[2010]{Primary 14E30; Secondary 14J10.}
\keywords{minimal model program, abundance theorem, log canonical 
ring}
\date{2011/8/18, version 4.24}
\address{Department of Mathematics, Faculty of 
Science, Kyoto University, 
Kyoto 606-8502 Japan} 
\email{fujino@math.kyoto-u.ac.jp}
\newcommand{\Proj}[0]{{\operatorname{Proj}}}
\newcommand{\Hom}[0]{{\operatorname{Hom}}}
\newcommand{\Bs}[0]{{\operatorname{Bs}}}
\newcommand{\Nklt}[0]{{\operatorname{Nklt}}}
\newcommand{\Nlc}[0]{{\operatorname{Nlc}}}
\newcommand{\Exc}[0]{{\operatorname{Exc}}}
\newcommand{\Supp}[0]{{\operatorname{Supp}}}
\newcommand{\Pic}[0]{{\operatorname{Pic}}}
\newcommand{\mult}[0]{{\operatorname{mult}}}
\newcommand{\Spec}[0]{{\operatorname{Spec}}}

\newtheorem{thm}{Theorem}[section]
\newtheorem{lem}[thm]{Lemma}
\newtheorem{cor}[thm]{Corollary}
\newtheorem{prop}[thm]{Proposition}

\theoremstyle{definition}
\newtheorem{ex}[thm]{Example}
\newtheorem{defn}[thm]{Definition}
\newtheorem{rem}[thm]{Remark}
\newtheorem*{ack}{Acknowledgments}      
         
\newtheorem{say}{Step}
\newtheorem{sa}[thm]{}
\begin{document}
\bibliographystyle{amsalpha+}

\maketitle

\begin{abstract}
We discuss the log minimal model theory for log surfaces. 
We show that the log minimal model program, 
the finite generation of log canonical rings, and 
the log abundance theorem for log surfaces hold 
true under assumptions weaker than the usual 
framework of the log minimal model theory. 
\end{abstract} 

\tableofcontents

\section{Introduction}\label{sec1}

We discuss the log minimal model theory for 
log surfaces. 
This paper completes Fujita's results 
on the semi-ampleness of 
semi-positive parts of Zariski decompositions of log canonical 
divisors and the finite generation of log canonical 
rings for smooth 
projective log surfaces in \cite{fujita} and 
the log minimal model program for projective log canonical 
surfaces discussed by Koll\'ar and 
Kov\'acs in \cite{koko}. 
We show that 
the log minimal model program for surfaces works and 
the log abundance theorem and the finite generation of 
log canonical rings for surfaces 
hold 
true under assumptions weaker than 
the usual framework of the log minimal model theory 
(cf.~Theorems \ref{thm32}, \ref{43}, 
and \ref{aban}). 

The log minimal model program 
works for $\mathbb Q$-factorial log surfaces and 
log canonical surfaces by our new cone and contraction 
theorem for log varieties (cf.~\cite[Theorem 1.1]{fujino2}), 
which is the culmination of the works of several authors. 
By our log minimal model program for log surfaces, 
Fujita's results in \cite{fujita} are clarified and 
generalized. 
In \cite{fujita}, Fujita treated a pair $(X, \Delta)$ where 
$X$ is a smooth projective surface and 
$\Delta$ is a boundary $\mathbb Q$-divisor on $X$ without any 
assumptions on singularities of the pair $(X, \Delta)$. 
We note that our log minimal model program discussed in this paper 
works for such pairs (cf.~Theorem \ref{thm32}). It is not necessary to assume that 
$(X, \Delta)$ is log canonical. 

Roughly speaking, 
we will prove the following theorem in this paper. 
Case (A) in Theorem \ref{main-ka} is new. 

\begin{thm}[{cf.~Theorems \ref{thm32} and \ref{rdiv}}]\label{main-ka}
Let $X$ be a normal projective surface defined over $\mathbb C$ 
and let $\Delta$ be an effective $\mathbb R$-divisor 
on $X$ such that 
every coefficient of $\Delta$ is less than or equal to one. 
Assume that one of the following conditions holds{\em{:}} 
\begin{itemize}
\item[(A)] $X$ is $\mathbb Q$-factorial, or 
\item[(B)] $(X, \Delta)$ is log canonical. 
\end{itemize}
Then we can run the log minimal model program 
with respect to $K_X+\Delta$ and 
obtain a sequence of extremal contractions  
$$
(X, \Delta)=(X_0, \Delta_0)
\overset{\varphi_0}\to (X_1, \Delta_1)\overset{\varphi_1}\to 
\cdots\overset{\varphi_{k-1}}\to (X_k, \Delta_k)=(X^*, \Delta^*)
$$ 
such that 
\begin{itemize}
\item[(1)] {\em{(Minimal model)}} $K_{X^*}+\Delta^*$ is semi-ample 
if $K_X+\Delta$ is pseudo-effective, and 
\item[(2)] {\em{(Mori fiber space)}} there is a morphism 
$g:X^*\to C$ such that 
$-(K_{X^*}+\Delta^*)$ is $g$-ample, $\dim C<2$, and 
the relative Picard number $\rho (X^*/C)=1$, 
if $K_X+\Delta$ is not pseudo-effective. 
\end{itemize}
We note that, in {\em{Case (A)}}, 
we do not assume that $(X, \Delta)$ is log canonical. 
We also note that $X_i$ is $\mathbb Q$-factorial for every $i$ in {\em{Case (A)}} 
and that $(X_i, \Delta_i)$ is log canonical 
for every $i$ in {\em{Case (B)}}. 
Moreover, in both cases, $X_i$ has only rational singularities for every $i$ if so does 
$X$ {\em{(cf.~Proposition \ref{2626})}}. 
\end{thm}

As a special case of Theorem \ref{main-ka},
we obtain a generalization of Fujita's result in \cite{fujita}, where
$X$ is assumed to be smooth.

\begin{cor}[{cf.~\cite{fujita}}]
Let $X$ be a normal projective surface defined over $\mathbb C$ and
let $\Delta$ be an effective $\mathbb Q$-divisor on $X$ such that
every coefficient of $\Delta$ is less than or equal to one.
Assume that $X$ is $\mathbb Q$-factorial and
$K_X+\Delta$ is pseudo-effective.
Then the semi-positive part of the Zariski decomposition of $K_X+\Delta$ is semi-ample. 
In particular, if $K_X+\Delta$ is nef, then it is semi-ample. 
\end{cor}

The following result is a corollary of Theorem \ref{main-ka}. 
It is because $X$ is $\mathbb Q$-factorial if $X$ has only rational singularities. 

\begin{cor}[{cf.~Corollary \ref{44}}]\label{1212}  
Let $X$ be a projective surface with only rational singularities. 
Then the canonical ring 
$$
R(X)=\bigoplus_{m\geq 0} H^0(X, \mathcal O_X(mK_X))
$$ 
is a finitely generated $\mathbb C$-algebra. 
\end{cor}

Furthermore, if $K_X$ is big in Corollary \ref{1212}, 
then we can prove 
that the {\em{canonical model}} 
$$Y=\Proj \bigoplus_{m\geq 0}H^0(X, \mathcal O_X(mK_X))$$ of $X$ 
has only rational singularities (cf.~Theorem \ref{atara}). 
Therefore, the notion of rational singularities is appropriate 
to the minimal model theory for 
log surfaces. 

We note that the general classification theory of 
algebraic surfaces is due essentially to 
the Italian school, and has been worked out in detail by Kodaira, 
in Shafarevich's seminar, and so on. The theory of 
log surfaces was studied by Iitaka, 
Kawamata, Miyanishi, Sakai, Fujita, and many others. 
See, for example, \cite{miyanishi} and \cite{sakai-ne}. 
Our viewpoint is more minimal-model-theoretic 
than any other works. We do not use the notion of 
{\em{Zariski decomposition}} in this paper 
(see Remark \ref{Z}). 

Let us emphasize the major differences 
between traditional 
arguments for log and normal surfaces (cf.~\cite{miyanishi}, 
\cite{sakai1}, and \cite{sakai-ne}) 
and our new framework discussed in this paper. 

\begin{sa}[Intersection pairing in the sense of Mumford]
Let $X$ be a normal projective surface and 
let $C_1$ and $C_2$ be curves on 
$X$. 
It is well known that 
we can define the intersection number 
$C_1\cdot C_2$ in the sense of Mumford without assuming 
that $C_1$ or $C_2$ is $\mathbb Q$-Cartier. 
However, in this paper, 
we only consider the intersection 
number $C_1\cdot C_2$ under the assumption that 
$C_1$ or $C_2$ is $\mathbb Q$-Cartier. 
This is a key point of the minimal model 
theory for surfaces 
from the viewpoint of 
Mori theory. 
\end{sa}

\begin{sa}[Contraction theorems by Grauert and Artin] 
Let $X$ be a normal projective surface and 
let $C_1, \cdots, C_n$ be irreducible curves on $X$ such that 
the intersection matrix $(C_i\cdot C_j)$ is negative definite. 
Then we have a  contraction 
morphism 
$f:X\to Y$ 
which contracts $\bigcup _i C_i$ to a finite number of 
normal points. 
It is a well known and very 
powerful contraction theorem which follows 
from results by Grauert and Artin (see, for example, 
\cite[Theorem 14.20]{bad}). 
In this paper, we do not use this type of 
contraction theorem. 
A disadvantage of the above contraction theorem is that 
$Y$ is not always projective. 
In general, $Y$ is only an algebraic space. 
Various experiences show that 
$Y$ sometimes has pathological 
properties.  
We only consider contraction morphisms associated 
to negative extremal rays of the Kleiman--Mori cone 
$\overline {NE}(X)$. 
In this case, $Y$ is necessarily projective. 
It is very natural from the viewpoint of the higher dimensional 
log minimal model program. 
\end{sa}

\begin{sa}[Zariski decomposition] 
Let $X$ be a smooth 
projective surface 
and let $D$ be a pseudo-effective 
divisor on $X$. 
Then we can decompose 
$D$ as follows. 
$$
D=P+N
$$
The negative part $N$ is an effective $\mathbb Q$-divisor and either 
$N=0$ or the intersection matrix of the irreducible 
components of $N$ is negative definite, and the semi-positive 
part $P$ is nef and the intersection of $P$ with 
each irreducible component of $N$ is zero. 
The Zariski decomposition 
played crucial roles in the studies of 
log and normal surfaces. 
In this paper, we do not use Zariski decomposition. 
Instead, we run the log minimal model program 
because we are mainly interested in adjoint divisors 
$K_X+\Delta$ and have a 
powerful framework of the log minimal model program. 
In our case, if $K_X+\Delta$ is pseudo-effective, 
then we have a contraction morphism $f:X\to X'$ such that 
\begin{align*}
K_X+\Delta=f^*(K_{X'}+\Delta')+E\tag{$\spadesuit$}
\end{align*}
where 
$K_{X'}+\Delta'$ is nef, and $E$ is effective and 
$f$-exceptional. Of course, $(\spadesuit)$ is the 
Zariski decomposition of $K_X+\Delta$. 
We think that it is more natural and easier to 
treat $K_{X'}+\Delta'$ on $X'$ than 
$f^*(K_X+\Delta)$ on $X$. 
\end{sa}

\begin{sa}[On Kodaira type vanishing theorems]
Let $X$ be a smooth projective surface and 
let $D$ be a simple normal crossing 
divisor on $X$. 
In the traditional arguments, $\mathcal O_X(K_X+D)$ 
was recognized to be $\Omega^2_X(\log D)$. 
For our vanishing theorems which play important roles in this 
paper, we have to 
recognize $\mathcal O_X(K_X+D)$ as 
$\mathcal H om _{\mathcal O_X}(\mathcal O_X(-D), \mathcal O_X(K_X))$ 
and $\mathcal O_X(-D)$ as the $0$-th term of $\Omega^{\bullet}_X(\log D)\otimes 
\mathcal O_X(-D)$. 
For details, 
see \cite[Section 5]{fujino2}, 
\cite[Chapter 2]{fujino3}, and \cite{fujino6}. 
The reader can find our philosophy of vanishing theorems for 
the log minimal model program in \cite[Section 3]{fujino2}. 
\end{sa}

\begin{sa}[$\mathbb Q$-factoriality]
In our framework, $\mathbb Q$-factoriality will play crucial 
roles. 
For surfaces, $\mathbb Q$-factoriality 
seems to be more useful than we expected. 
See Lemma \ref{lem-fac} 
and Theorem \ref{thm53}. 
The importance of $\mathbb Q$-factoriality will be 
clarified in the minimal model theory of log surfaces 
in positive characteristic. 
For details, see \cite{tanaka}.
\end{sa}

Anyway, this paper gives a new framework for 
the study of log and normal surfaces. 

We summarize the contents of this paper. 
Section \ref{sec2} collects some preliminary results. 
In Section \ref{sec3}, we discuss the log minimal model program for 
log surfaces. 
It is a direct consequence of the cone and contraction theorem 
for log varieties (cf.~\cite[Theorem 1.1]{fujino2}). 
In Section \ref{sec4}, we show the finite generation of log canonical 
rings for log surfaces. More precisely, 
we prove a special case of the log abundance theorem 
for log surfaces. 
In Section \ref{sec5}, 
we treat the non-vanishing theorem for log surfaces. 
It is an important step of the log abundance theorem for log surfaces. 
In Section \ref{sec6}, 
we prove the log abundance theorem for log surfaces. 
It is a generalization of Fujita's main result in \cite{fujita}. 
Section \ref{sec7} is a supplementary section. 
We prove the finite generation of log canonical rings 
and the log abundance theorem for log surfaces in the relative setting. 
In Section \ref{new-sec8}, 
we generalize the relative log abundance theorem 
in Section \ref{sec7} for $\mathbb R$-divisors. 
Consequently, Theorem \ref{main-ka} also holds in the relative setting. 
In Section \ref{sec8}:~Appendix, we prove the 
base point free theorem for log surfaces 
in full generality (cf.~Theorem \ref{bpf}), 
though it is not necessary for the log minimal 
model theory for log surfaces discussed 
in this paper. 
It generalizes Fukuda's base point free theorem for 
log canonical surfaces (cf.~\cite[Main Theorem]{fukuda}). 
Our proof is different from 
Fukuda's and depends on the theory of {\em{quasi-log 
varieties}} (cf.~\cite{ambro}, \cite{fujino3}, 
and \cite{fujino-s}). 

We will work over $\mathbb C$, the complex number field, 
throughout this paper. 
Our arguments heavily depend on a Kodaira type vanishing theorem 
(cf.~\cite{fujino2}). 
So, we can not directly apply them in positive characteristic. 
We note that \cite{fujita} and \cite{koko} 
treat algebraic surfaces defined over an algebraically 
closed field in {\em{any}} characteristic. 
Recently, Hiromu Tanaka establishes the minimal model theory 
of log surfaces in positive characteristic 
(see \cite{tanaka}). 
Simultaneously, he slightly simplifies and generalizes 
some arguments in this paper (cf.~Theorem \ref{thm53} and Remark \ref{rem-ta}). 
Consequently, all the results in this paper hold 
over any algebraically closed field of characteristic zero. 
We have to be careful when we 
use the Lefschetz principle because $\mathbb Q$-factoriality 
is not necessarily preserved by field extensions (cf.~Remark \ref{rem65}). 

\begin{ack}
The author would like to 
thank Professors Takao Fujita and Fumio Sakai. 
He was partially supported by The Inamori Foundation and by the 
Grant-in-Aid for Young Scientists (A) $\sharp$20684001 from JSPS. 
He thanks Takeshi Abe and Yoshinori 
Gongyo for comments and discussions. 
He also thanks Professor Shigefumi Mori for useful 
comments, discussions, and warm encouragement. 
Finally, he thanks Hiromu Tanaka for stimulating discussions. 
\end{ack}

\section{Preliminaries}\label{sec2}

We collect some basic definitions and results. 
We will freely use the notation and terminology 
in \cite{km} and \cite{fujino2} throughout this paper.

\begin{sa}[$\mathbb Q$-divisors and $\mathbb R$-divisors] 
Let $X$ be a normal variety. 
For an $\mathbb R$-divisor $D=\sum _{j=1}^rd_j D_j$ on $X$ such that 
$D_j$ is a prime divisor for every $j$ and 
$D_i\ne D_j$ for $i\ne j$, 
we define the {\em{round-down}} $\llcorner D\lrcorner 
=\sum _{j=1}^r\llcorner d_j\lrcorner 
D_j$ (resp.~{\em{round-up}} 
$\ulcorner D\urcorner=\sum _{j=1}^{r}\ulcorner 
d_j\urcorner D_j$), where 
for every real number $x$, $\llcorner 
x\lrcorner$ (resp.~$\ulcorner x\urcorner$) 
is the integer defined by $x-1<\llcorner x\lrcorner 
\leq x$ (resp.~$\ulcorner x\urcorner=-\llcorner -x\lrcorner$). 
The {\em{fractional part}} $\{D\}$ of $D$ denotes 
$D-\llcorner D\lrcorner$. We define 
$$
D^{>a}=\sum _{d_j>a}d_j D_j, \quad D^{<a}=\sum _{d_j<a}d_jD_j 
$$
and 
$$
D^{=a}=\sum _{d_j=a}d_j D_j=a\sum _{d_j=a}D_j 
$$
for any real number $a$. 
We call $D$ a {\em{boundary}} 
$\mathbb R$-divisor if $0\leq d_j \leq 1$ for 
every $j$. 
We note that 
$\sim _{\mathbb Q}$ (resp.~$\sim _{\mathbb R}$) denotes 
the {\em{$\mathbb Q$-linear equivalence}} (resp.~{\em{$\mathbb R$-linear 
equivalence}}) of $\mathbb Q$-divisors (resp.~$\mathbb R$-divisors). 
Of course, $\sim$ (resp.~$\equiv$) 
denotes the usual {\em{linear equivalence}} (resp.~{\em{numerical 
equivalence}}) of divisors. 

Let $f:X\to Y$ be a morphism and let $B$ be a Cartier divisor on $X$. 
We say that $B$ is {\em{linearly $f$-trivial}} 
(denoted by $B\sim _f 0$) 
if and only if there is a Cartier divisor $B'$ on $Y$ such that 
$B\sim f^*B'$. 
Two $\mathbb R$-Cartier $\mathbb R$-divisors $B_1$ and $B_2$ on $X$ are 
called {\em{numerically $f$-equivalent}} 
(denoted by $B_1\equiv _f B_2$) if and 
only if $B_1\cdot C=B_2\cdot C$ for every curve $C$ such that 
$f(C)$ is a point. 

We say that $X$ is {\em{$\mathbb Q$-factorial}} 
if every prime Weil divisor on $X$ is $\mathbb Q$-Cartier. 
The following lemma is well known. 

\begin{lem}[Projectivity]\label{222} 
Let $X$ be a complete normal $\mathbb Q$-factorial 
algebraic surface. 
Then 
$X$ is projective. 
More precisely, 
a normal $\mathbb Q$-factorial algebraic surface 
is always quasi-projective. 
\end{lem}
\begin{proof}
Let $X$ be a normal $\mathbb Q$-factorial algebraic surface. 
Then it is easy to construct a complete normal $\mathbb Q$-factorial 
algebraic surface $\overline X$ which contains $X$ as a Zariski open subset. 
It is because $X$ has only isolated singularities. 
So, from now on, we treat a complete normal $\mathbb Q$-factorial 
algebraic surface. 
Let $f:Y\to X$ be a projective birational morphism from 
a smooth projective surface $Y$. 
Let $H$ be an effective general ample Cartier divisor on $Y$. 
We consider the effective $\mathbb Q$-Cartier 
Weil divisor $A=f_*H$ on $X$. 
Then $A\cdot C=H\cdot f^*C>0$ for every 
curve $C$ on $X$. 
Therefore, $A$ is ample by Nakai's criterion. 
Thus, $X$ is projective. 
\end{proof}

By the following example, we know that 
$\mathbb Q$-factoriality of a surface is weaker than 
the condition that the surface has only rational singularities. 

\begin{ex}
We consider 
$$
X=\Spec\,  \mathbb C[X_1, X_2, X_3]/(X_1^{e_1}+X_2^{e_2}+X_3^{e_3})
$$
where $e_1, e_2$, and $e_3$ are positive integers such that 
$1<e_2<e_2<e_3$ and 
$(e_i, e_j)=1$ for $i\ne j$. 
Then $X$ is factorial, that is, every Weil divisor on $X$ is Cartier 
(see, for example, \cite[Theorem 5.1]{morimori}). 
If $(e_1, e_2, e_3)=(2, 3, 5)$, 
then $X$ has a singular point of $E_8$ type. Therefore, 
$X$ has a rational Gorenstein singularity. 
If $(e_1, e_2, e_3)\ne (2, 3, 5)$, then 
the singularity of $X$ is not rational. 
Therefore, there are many normal ($\mathbb Q$-)factorial 
surfaces whose singularities are not rational. 
\end{ex}
\end{sa} 

\begin{sa}[Singularities of pairs]
Let $X$ be a normal variety and let $\Delta$ be 
an effective $\mathbb R$-divisor on $X$ 
such that $K_X+\Delta$ is $\mathbb R$-Cartier. 
Let $f:Y\to X$ be a resolution such that 
$\Exc (f)\cup f^{-1}_*\Delta$ has simple normal crossing support, 
where $\Exc (f)$ is the {\em{exceptional locus}} 
of $f$ and $f^{-1}_*\Delta$ is 
the {\em{strict transform}} of $\Delta$ on $Y$. 
We can write 
$$
K_Y=f^*(K_X+\Delta)+\sum _i a_i E_i.  
$$ 
We say that $(X, \Delta)$ is {\em{log canonical}} ({\em{lc}}, for short)  
if $a_i\geq -1$ for every $i$. 
We say that $(X, \Delta)$ is {\em{Kawamata log terminal}} 
({\em{klt}}, for short) if $a_i>-1$ for every $i$. 
We usually write $a_i=a(E_i, X, \Delta)$ and call it 
the {\em{discrepancy coefficient}} of $E_i$ with respect to 
$(X, \Delta)$. 
We note that $\Nklt (X, \Delta)$ (resp.~$\Nlc (X, \Delta)$) 
denotes the image of $\sum _{a_i\leq -1} E_i$ (resp.~$\sum _{a_i<-1}E_i$) and 
is called the {\em{non-klt locus}} (resp.~{\em{non-lc locus}}) of $(X, \Delta)$. 
If there exist a resolution $f:Y\to X$ and a divisor $E$ on $Y$ such 
that $a(E, X, \Delta)=-1$ and 
that $f(E)\not\subset \Nlc (X, \Delta)$, then 
$f(E)$ is called a {\em{log canonical center}} 
({\em{lc center}}, for short) with respect to $(X, \Delta)$. 
If there exist a resolution $f:Y\to X$ and a divisor $E$ on $Y$ such 
that $a(E, X, \Delta)\leq -1$, 
then 
$f(E)$ is called a {\em{non-klt center}} with respect to $(X, \Delta)$. 

When $X$ is a surface, the notion of {\em{numerically 
log canonical}} and {\em{numerically dlt}} is sometimes 
useful. 
See \cite[Notation 4.1]{km} and Proposition \ref{pro} below. 
\end{sa}

\begin{sa}[Kodaira dimension and numerical Kodaira dimension]
We note that $\kappa$ (resp.~$\nu$) denotes 
the {\em{Iitaka--Kodaira dimension}} 
(resp.~{\em{numerical Kodaira dimension}}). 

Let $X$ be a normal projective variety, $D$ a $\mathbb Q$-Cartier 
$\mathbb Q$-divisor on $X$, 
and $n$ a positive integer such that 
$nD$ is Cartier. 
By definition, $\kappa (X, D)=-\infty$ if and only if 
$h^0(X, \mathcal O_X(mnD))=0$ for 
every $m>0$, and 
$\kappa (X, D)=k>-\infty$ if and only if 
$$
0<\underset{m>0}{\lim\sup} 
\frac{h^0(X, \mathcal O_X(mnD))}{m^k}<\infty. 
$$ 
We see that 
$\kappa (X, D)\in \{-\infty, 0, 1, \cdots, \dim X\}$. 
If $D$ is nef, then 
$$
\nu(X, D)=\max \{ e\in \mathbb Z_{\geq 0} \, | \, 
D^e \ \text{is not numerically zero}\}. 
$$ 
We say that $D$ is {\em{abundant}} if 
$\nu(X, D)=\kappa (X, D)$. 

Let $Y$ be a projective irreducible 
variety and let $B$ be a $\mathbb Q$-Cartier $\mathbb Q$-divisor 
on $Y$. 
We say that $B$ is {\em{big}} if $\nu^*B$ is {\em{big}}, 
that is, $\kappa (Z, \nu^*B)=\dim Z$, 
where $\nu:Z\to Y$ is the normalization of $Y$, 
\end{sa}
\begin{sa}[Nef dimension]\label{2525}
Let $L$ be a nef $\mathbb Q$-Cartier $\mathbb Q$-divisor 
on a normal projective variety 
$X$. Then $n(X, L)$ denotes the {\em{nef dimension}} of $L$. 
It is well known that $$\kappa (X, L)\leq \nu (X, L)\leq n(X, L). $$ 
For details, see \cite{8a}. 
We will use the {\em{reduction map}} associated to 
$L$ in Section \ref{sec6}. 

Let us quickly recall the reduction map and the nef dimension in \cite{8a}. 
By \cite[Theorem 2.1]{8a}, for a 
nef $\mathbb Q$-Cartier $\mathbb Q$-divisor 
$L$ on $X$, 
we can construct an almost holomorphic, dominant rational map 
$f:X\dashrightarrow Y$ with 
connected fibers, called a {\em{reduction map}} associated to 
$L$ such that 
\begin{itemize}
\item[(i)] $L$ is numerically trivial on all compact fibers 
$F$ of $f$ with $\dim F=\dim X-\dim Y$, and 
\item[(ii)] for every general point $x\in X$ and 
every irreducible curve $C$ passing 
through $x$ with $\dim f(C)>0$, we have $L\cdot C>0$. 
\end{itemize} 
The map $f$ is unique up to birational 
equivalence of $Y$. 
We define the {\em{nef dimension}} of $L$ 
as follows (cf.~\cite[Definition 2.7]{8a}):
$$
n(X, L):=\dim Y. 
$$   
\end{sa}
\begin{sa}[Non-lc ideal sheaves] 
The ideal sheaf 
$\mathcal J_{NLC}(X, \Delta)$ denotes the {\em{non-lc ideal sheaf}} 
associated to the pair $(X, \Delta)$. 
More precisely, let 
$X$ be a normal variety and let $\Delta$ be an 
effective $\mathbb R$-divisor on $X$ such 
that $K_X+\Delta$ is $\mathbb R$-Cartier. 
Let $f:Y\to X$ be a resolution 
such that 
$
K_Y+\Delta_Y=f^*(K_X+\Delta) 
$ and that 
$\Supp \Delta_Y$ is simple normal crossing. 
Then we have  
$$
\mathcal J_{NLC}(X, \Delta)=f_*\mathcal O_Y(-\llcorner 
\Delta_Y\lrcorner +\Delta^{=1}_Y)\subset \mathcal O_X. 
$$
For details, see, for example, \cite[Section 7]{fujino2}, 
\cite{fujino8}, or \cite{fst}. 
We note that 
$$
\mathcal J(X, \Delta)=
f_*\mathcal O_Y(-\llcorner \Delta_Y\lrcorner)\subset 
\mathcal O_X
$$ is 
the {\em{multiplier ideal sheaf}} associated to the pair $(X, \Delta)$. 
\end{sa}

\begin{sa}[a Kodaira type vanishing theorem]\label{vani} 
Let $f:X\to Y$ be a birational morphism from a smooth 
projective variety $X$ to a normal 
projective variety $Y$. 
Let $\Delta$ be a boundary $\mathbb Q$-divisor on $X$ such that 
$\Supp \Delta$ is a simple normal crossing divisor and 
let $L$ be a Cartier divisor on $X$. 
Assume that 
$$
L-(K_X+\Delta)\sim _{\mathbb Q}f^*H, 
$$ 
where $H$ is a nef and big $\mathbb Q$-Cartier 
$\mathbb Q$-divisor on $Y$ such that $H|_{f(C)}$ is big for every 
lc center $C$ of the pair $(X, \Delta)$. 
Then 
we obtain 
$$
H^i(Y, R^jf_*\mathcal O_X(L))=0
$$ 
for every $i>0$ and $j\geq 0$. 
It is a special case of \cite[Theorem 2.47]{fujino3}, 
which is the culmination of the works of several authors. 
We recommend \cite{fujino6} as an introduction to 
new vanishing theorems. 
\end{sa}

\begin{sa}Let $\Lambda$ be a linear system. 
Then $\Bs\Lambda$ denotes the {\em{base locus}} of $\Lambda$.  
\end{sa}

\section{Minimal model program for log surfaces}\label{sec3} 

Let us recall the notion of {\em{log surfaces}}. 

\begin{defn}[Log surfaces]
Let $X$ be a normal algebraic surface and let $\Delta$ 
be a boundary $\mathbb R$-divisor on 
$X$ such that $K_X+\Delta$ is $\mathbb R$-Cartier. 
Then the pair $(X, \Delta)$ is called a {\em{log surface}}. 
We recall that a {\em{boundary}} $\mathbb R$-divisor is 
an effective $\mathbb R$-divisor whose coefficients 
are less than or equal to one. 
\end{defn}

We note that we assume nothing on singularities of $(X, \Delta)$. 

From now on, we discuss the log minimal model program for log surfaces. 
The following cone and contraction 
theorem is a special case of \cite[Theorem 1.1]{fujino2}. 
For details, see \cite{fujino2}. 

\begin{thm}[{cf.~\cite[Theorem 1.1]{fujino2}}]\label{thm-cone}
Let $(X, \Delta)$ be a log surface and 
let $\pi:X\to S$ be a projective 
morphism onto an algebraic variety $S$. 
Then we have 
$$\overline {NE}(X/S)=\overline {NE}(X/S)_{K_X+\Delta\geq 0} 
+\sum R_j$$ 
with the following properties.  
\begin{itemize}
\item[(1)] 
$R_j$ is a $(K_X+\Delta)$-negative 
extremal ray of $\overline {NE}(X/S)$ for every $j$. 
\item[(2)] Let $H$ be a $\pi$-ample $\mathbb R$-divisor 
on $X$. 
Then there are only finitely many $R_j$'s included in 
$(K_X+\Delta+H)_{<0}$. In particular, 
the $R_j$'s are discrete in the half-space 
$(K_X+\Delta)_{<0}$. 
\item[(3)] 
Let $R$ be a $(K_X+\Delta)$-negative extremal 
ray of $\overline {NE}(X/S)$.  
Then there exists a contraction morphism $\varphi_R:X\to Y$ 
over $S$ with the following properties. 
\begin{itemize}
\item[(i)] 
Let $C$ be an integral curve on $X$ such that 
$\pi(C)$ is a point. 
Then $\varphi_R(C)$ is a point if and 
only if 
$[C]\in R$. 
\item[(ii)] $\mathcal O_Y\simeq (\varphi_R)_*\mathcal O_X$.  
\item[(iii)] Let $L$ be a line bundle on $X$ such 
that $L\cdot C=0$ for 
every curve $C$ with $[C]\in R$. 
Then there exists a line bundle $L_Y$ on $Y$ such 
that $L\simeq \varphi^*_RL_Y$. 
\end{itemize}
\end{itemize}
\end{thm}

A key point is that the non-lc locus of a log surface $(X, \Delta)$ 
is zero-dimensional. 
So, there are no curves contained in the non-lc locus 
of $(X, \Delta)$. 
We will prove that $R_j$ in Theorem \ref{thm-cone} (1) is spanned 
by a rational curve $C_j$ with 
$-(K_X+\Delta)\cdot C_j\leq 3$ in Proposition \ref{2727} 
below. 

By Theorem \ref{thm-cone}, we can run the log minimal model 
program for log surfaces under some mild assumptions. 

\begin{thm}[Minimal model program for log surfaces]\label{thm32}
Let $(X, \Delta)$ be a log surface and 
let $\pi:X\to S$ be a projective morphism onto an algebraic 
variety $S$. 
We assume one of the following conditions{\em{:}} 
\begin{itemize}
\item[(A)] $X$ is $\mathbb Q$-factorial. 
\item[(B)] $(X, \Delta)$ is log canonical. 
\end{itemize}
Then, by {\em{Theorem \ref{thm-cone}}}, 
we can run the log minimal model program over $S$ with respect to 
$K_X+\Delta$.  
So, there is a sequence of at most $\rho (X/S)-1$ contractions 
$$
(X, \Delta)=(X_0, \Delta_0)\overset{\varphi_0}
\to (X_1, \Delta_1)\overset{\varphi_1}\to 
\cdots\overset{\varphi_{k-1}}\to (X_k, \Delta_k)=(X^*, \Delta^*)
$$
over $S$ such that one of the following holds{\em{:}}
\begin{itemize}
\item[(1)] {\em{(Minimal model)}} $K_{X^*}+\Delta^*$ is nef over $S$. 
In this case, $(X^*, \Delta^*)$ is called a {\em{minimal model}} of 
$(X, \Delta)$.  
\item[(2)] {\em{(Mori fiber space)}} There is a morphism $g:X^*\to C$ over 
$S$ such that $-(K_{X^*}+\Delta^*)$ is $g$-ample, $\dim C<2$, and 
$\rho (X^*/C)=1$. We sometimes call 
$g:(X^*, \Delta^*)\to C$ a {\em{Mori fiber space}}.  
\end{itemize}
We note that $X_i$ is $\mathbb Q$-factorial 
{\em{(}}resp.~$(X_i, \Delta_i)$ is 
lc{\em{)}} for every $i$ in 
{\em{Case (A)}} {\em{(}}resp.~{\em{(B)}}{\em{)}}. 
\end{thm}

\begin{proof}
It is obvious by Theorem \ref{thm-cone}. 
In Case (A), we can easily check 
that $X_i$ is $\mathbb Q$-factorial for 
every $i$ by the usual method (cf.~\cite[Proposition 3.36]{km}). 
In Case (B), we have to check that $(X_i, \Delta_i)$ is lc for 
$\Delta_i=\varphi_{i-1*}\Delta_{i-1}$. 
Since $-(K_{X_{i-1}}+\Delta_{i-1})$ is $\varphi_{i-1}$-ample, 
it is easy to see that $(X_i, \Delta_i)$ is 
numerically lc (cf.~\cite[Notation 4.1]{km}) 
by the negativity lemma. 
By Proposition \ref{pro} below, the pair 
$(X_i, \Delta_i)$ is log canonical. 
In particular, $K_{X_i}+\Delta_i$ is $\mathbb R$-Cartier. 
\end{proof}

As an application of Case (A) in Theorem \ref{thm32}, 
we obtain the following corollary. 

\begin{cor}\label{lem-q} 
Let $f:Y\to X$ be a projective birational morphism 
between normal surfaces. 
Let $\Delta_Y$ be an effective $\mathbb R$-divisor 
on $Y$ such that $\Supp \Delta_Y \subset \Exc(f)$ and $\llcorner 
\Delta_Y\lrcorner=0$. Assume that 
$Y$ is $\mathbb Q$-factorial 
and that $K_Y+\Delta_Y\equiv _f0$. Then 
$X$ is $\mathbb Q$-factorial. 
\end{cor}

\begin{proof}
We put $E=\Exc (f)$. 
We run the $(K_Y+\Delta _Y+\varepsilon E)$-minimal model program over $X$ where 
$\varepsilon$ is a small positive number such that 
$\llcorner \Delta_Y+\varepsilon E\lrcorner =0$. 
By the negativity lemma, the above minimal model 
program terminates at $X$. 
Therefore, $X$ is $\mathbb Q$-factorial 
by Theorem \ref{thm32} (A). 
\end{proof}

Let us contain \cite[Proposition 4.11]{km} 
for the reader's convenience. 
The statement (2) in the following proposition is missing 
in the English edition of \cite{km}. 
For definitions, see \cite[Notation 4.1]{km}. 

\begin{prop}[{cf.~\cite[Proposition 4.11]{km}}]\label{pro} 
We have the following two statements. 

{\em{(1)}} Let $(X, \Delta)$ be a numerically dlt pair. Then every 
Weil divisor on $X$ is $\mathbb Q$-Cartier, that is, 
$X$ is $\mathbb Q$-factorial. 

{\em{(2)}} Let $(X, \Delta)$ be a numerically lc pair. 
Then it is lc. 
\end{prop}
\begin{proof}
In both cases, if $\Delta\ne 0$, then $(X, 0)$ is numerically 
dlt by \cite[Corollary 4.2]{km} and we can reduce the problem 
to the case (1) 
with $\Delta=0$. 
Therefore, we may assume that $\Delta=0$ when we prove this proposition. 
Let $f:Y\to X$ be a minimal resolution and 
let $\Delta_Y$ be the $f$-exceptional $\mathbb Q$-divisor on $Y$ such 
that $K_Y+\Delta_Y\equiv _f 0$. 
Then $\Delta_Y\geq 0$ by \cite[Corollary 4.3]{km}. 

(1) We can apply Corollary \ref{lem-q} 
since $\llcorner \Delta_Y\lrcorner =0$. We note that 
we only used Case (A) of Theorem \ref{thm32} for the 
proof of Corollary \ref{lem-q}. 
See also the proof of \cite[Proposition 4.11]{km}. 

(2) We may assume that $(X, 0)$ is not numerically dlt, that is, 
$\llcorner \Delta_Y\lrcorner\ne 0$. 
By \cite[Theorem 4.7]{km}, $\{\Delta_Y\}$ is 
a simple normal crossing divisor. 
Since $-\llcorner \Delta_Y\lrcorner \equiv _f K_Y+\{\Delta_Y\}$, we 
have 
$$R^1f_*\mathcal O_Y(n(K_Y+\Delta_Y)-\llcorner \Delta_Y\lrcorner)=0$$ by 
the Kawamata--Viehweg vanishing theorem 
for $n\in \mathbb Z_{>0}$ such that $n\Delta_Y$ is a Weil divisor. 
Therefore, we obtain a surjection 
$$f_*\mathcal O_Y(n(K_Y+\Delta_Y))\twoheadrightarrow 
f_*\mathcal O_{\llcorner \Delta_Y\lrcorner}(n(K_Y+\Delta_Y)).$$ 
Therefore, if we check 
$$
n(K_Y+\Delta_Y)|_{\llcorner \Delta_Y\lrcorner}\sim 0, 
$$ 
then we obtain $n(K_Y+\Delta_Y)\sim _f 0$ and 
$nK_X=f_*(n(K_Y+\Delta_Y))$ is a Cartier divisor. 
This statement can be checked by \cite[Theorem 4.7]{km} as follows. 
By the classification, $\llcorner \Delta_Y\lrcorner$ is a cycle and 
$\Delta_Y=\llcorner \Delta_Y\lrcorner$ 
(cf.~\cite[Definition 4.6]{km}), 
or $\llcorner \Delta_Y\lrcorner$ is a simple normal 
crossing divisor consisting of rational curves and the dual graph is a tree. 
In the former case, we have $K_{\Delta_Y}\sim 0$. So, $n=1$ is sufficient. 
In the latter case, since $H^1(\mathcal O_{\llcorner \Delta_Y\lrcorner})=0$, 
$n(K_Y+\Delta_Y)|_{\llcorner \Delta_Y\lrcorner}\sim 0$ if we choose 
$n>0$ such that 
$n(K_Y+\Delta_Y)$ is a numerically trivial 
Cartier divisor (cf.~\cite[Theorem 4.13]{km}). 
\end{proof}

We give an important remark on rational singularities. 

\begin{rem} 
Let $X$ be an algebraic surface. 
If $X$ has only rational singularities, then it is well known that 
$X$ is $\mathbb Q$-factorial. 
Therefore, we can apply the log minimal model program in Theorem \ref{thm32} 
for pairs of surfaces with only rational 
singularities and boundary $\mathbb R$-divisors 
on them. We note that there are many two-dimensional rational singularities 
which are not lc. 

We take a rational non-lc surface singularity $P\in X$. 
Let $\pi:Z\to X$ be the index one cover of $X$. 
In this case, $Z$ is not log canonical 
or rational.  
\end{rem}

We note that our log minimal model program works for the class of surfaces with 
only rational singularities by the next proposition. 
It is similar to \cite[Proposition 2.71]{km}. 
It is mysterious that \cite[Proposition 2.71]{km} is also missing 
in the English edition of \cite{km}. 

\begin{prop}\label{2626}
Let $(X, \Delta)$ be a log surface and let $f:X\to Y$ be a projective surjective 
morphism onto a normal surface $Y$. Assume that $-(K_X+\Delta)$ is $f$-ample. 
Then $R^if_*\mathcal O_X=0$ for every $i>0$. 
Therefore, if $X$ has only rational singularities, then 
$Y$ also has only rational singularities. 
\end{prop}

\begin{proof}We consider the short exact sequence 
$$
0\to \mathcal J_{NLC}(X, \Delta)\to 
\mathcal O_X\to \mathcal O_X/\mathcal J_{NLC}(X, \Delta)\to 0, 
$$ 
where $\mathcal J_{NLC}(X, \Delta)$ is the non-lc ideal sheaf associated to 
the pair $(X, \Delta)$. 
By the vanishing theorem (cf.~\cite[Theorem 8.1]{fujino2}), we know 
$R^if_*\mathcal J_{NLC}(X, \Delta)=0$ for every $i>0$. 
Since $\Delta$ is a boundary $\mathbb R$-divisor, 
we have $\dim _{\mathbb C}\Supp (\mathcal O_X/\mathcal J_{NLC}(X, \Delta))=0$. 
So, we obtain 
$R^if_*(\mathcal O_X/\mathcal J_{NLC}(X, \Delta))=0$ for every $i>0$. 
Thus, $R^if_*\mathcal O_X=0$ for all $i>0$. 
\end{proof}

As a corollary, we can check the following result. 

\begin{prop}[Extremal rational curves]\label{2727}
Let $(X, \Delta)$ be a log surface and let $\pi:X\to S$ be a projective surjective 
morphism onto a variety $S$. 
Let $R$ be a $(K_X+\Delta)$-negative extremal ray. 
Then $R$ is spanned by a rational curve $C$ on $X$ such that 
$-(K_X+\Delta)\cdot C\leq 3$. 
Moreover, if $X\not\simeq \mathbb P^2$, then 
we can choose $C$ with $-(K_X+\Delta)\cdot C\leq 2$. 
\end{prop}

\begin{proof}We consider the extremal contraction 
$\varphi_R:X\to Y$ over $S$ associated to 
$R$. 
Let $f:Z\to X$ be the minimal resolution such that 
$K_Z+\Delta_Z=f^*(K_X+\Delta)$. 
Note that $\Delta_Z$ is effective. 
First, we assume that $Y$ is a point. 
Let $D$ be a general curve on $Z$. 
Then $D\cdot (K_Z+\Delta_Z)=D\cdot f^*(K_X+\Delta)<0$. 
Therefore, $\kappa (Z, K_Z)=-\infty$. 
If $X\simeq \mathbb P^2$, then 
the statement is obvious. 
So, we may assume that $X\not\simeq \mathbb P^2$. 
In this case, there exists a morphism $g:Z\to B$ onto 
a smooth curve $B$. Let $D$ be a general fiber of $g$. 
Then $D\simeq \mathbb P^1$ and $-(K_Z+\Delta_Z)\cdot D
=-f^*(K_X+\Delta)\cdot D\leq 2$. Thus, 
$C=f(D)\subset X$ has the desired properties. 
Next, we assume that  $Y$ is a curve. 
In this case, we take a general fiber of 
$\varphi_R\circ f:Z\to X\to Y$. 
Then, it gives a desired curve as in the previous case. 
Finally, we assume that $\varphi_R:X\to Y$ 
is birational. Let $E$ be an irreducible component of 
the exceptional locus of $\varphi_R$. 
We consider the short exact sequence 
$$
0\to \mathcal I_E\to \mathcal O_X\to \mathcal O_E\to 0,  
$$ 
where $\mathcal I_E$ is the defining ideal sheaf of $E$ on $X$. 
By Proposition \ref{2626}, 
$R^1\varphi_{R*}\mathcal O_X=0$. 
Therefore, $R^1\varphi_{R*}\mathcal O_E=H^1(E, \mathcal O_E)=0$. 
Thus, $E\simeq \mathbb P^1$. 
Let $F$ be the strict transform of 
$E$ on $Z$. 
Then the coefficient of $F$ in $\Delta_Z$ is $\leq 1$ and 
$F^2<0$. 
Therefore, $-f^*(K_X+\Delta)\cdot F=-(K_Z+\Delta_Z)\cdot F\leq 2$. 
This means that $-(K_X+\Delta)\cdot E\leq 2$ and 
$E$ spans $R$. 
\end{proof}

We note the following 
easy result. 

\begin{prop}[Uniqueness]\label{p38} 
Let $(X, \Delta)$ be a log surface and 
let $\pi: X\to S$ be a projective 
morphism onto a variety $S$ as in {\em{Theorem 
\ref{thm32}}}. 
Let $(X^*, \Delta^*)$ and $(X^\dag, \Delta^\dag)$ be minimal 
models of $(X, \Delta)$ over $S$. Then 
$(X^*, \Delta^*)\simeq (X^\dag, \Delta^\dag)$ over $S$. 
\end{prop}

\begin{proof}
We consider 
$$
K_X+\Delta=f^*(K_{X^*}+\Delta^*)+E, 
$$
and 
$$
K_X+\Delta=g^*(K_{X^\dag}+\Delta^\dag)+F, 
$$ 
where $f:X\to X^*$ and $g:X\to X^\dag$. We note 
that $\Supp E=\Exc (f)$ and 
$\Supp F=\Exc (g)$. 
By the negativity lemma, we 
obtain $E=F$. 
Therefore, $(X^*, \Delta^*)\simeq (X^\dag, \Delta^\dag)$ 
over $S$.  
\end{proof}
 
We close this section with a remark on the {\em{Zariski decomposition}}. 
 
\begin{rem}\label{Z} 
Let $(X, \Delta)$ be a projective 
log surface such that $K_X+\Delta$ is $\mathbb Q$-Cartier and 
pseudo-effective. 
Assume that $(X, \Delta)$ is log canonical 
or $X$ is $\mathbb Q$-factorial. 
Then there exists the unique 
minimal model $(X^*, \Delta^*)$ of $(X, \Delta)$ by Theorem \ref{thm32} 
and Proposition \ref{p38}. 
Let $f:X\to X^*$ be the natural morphism. 
Then we can write 
$$
K_X+\Delta=f^*(K_{X^*}+\Delta^*)+E,  
$$ 
where $E$ is an effective $\mathbb Q$-divisor such that 
$\Supp E=\Exc (f)$. 
It is easy to see that 
$f^*(K_{X^*}+\Delta^*)$ (resp.~$E$) 
is the semi-positive (resp.~negative) part of 
the Zariski decomposition 
of $K_X+\Delta$. 
By Theorem \ref{aban} below, 
the semi-positive part $f^*(K_{X^*}+\Delta^*)$ of 
the Zariski decomposition 
of $K_X+\Delta$ is semi-ample. 
\end{rem}
 
\section{Finite generation of log canonical rings}\label{sec4}

In this section, we prove that the log canonical ring 
of a $\mathbb Q$-factorial projective log surface is finitely generated. 

First, we prove a special case of the log abundance conjecture for 
log surfaces. Our proof heavily depends on a Kodaira 
type vanishing theorem. 
 
\begin{thm}[Semi-ampleness]\label{41} 
Let $(X, \Delta)$ be a $\mathbb Q$-factorial projective log surface. 
Assume that $K_X+\Delta$ is nef and big and that $\Delta$ is 
a $\mathbb Q$-divisor. 
Then $K_X+\Delta$ is semi-ample. 
\end{thm} 

\begin{proof}
If $(X, \Delta)$ is klt, then $K_X+\Delta$ is semi-ample by 
the Kawamata--Shokurov base point free theorem. 
Therefore, we may assume that $(X, \Delta)$ is not klt. 
We divide the proof into several steps. 

\setcounter{say}{-1} 
\begin{say}\label{t0}
Let $\llcorner \Delta\lrcorner=\sum _i C_i$ be 
the irreducible decomposition. 
We put 
$$
A=\sum _{C_i\cdot (K_X+\Delta)=0}C_i \quad \text{and}\quad  
B=\sum _{C_i\cdot (K_X+\Delta)>0}C_i. 
$$
Then $\llcorner \Delta \lrcorner =A+B$. We note that 
$(C_i)^2<0$ if $C_i\cdot (K_X+\Delta)=0$ by 
the Hodge index theorem. 
We can decompose $A$ into 
the connected components as follows: 
$$
A=\sum _j A_j. 
$$
\end{say}

First, let us recall the following well-known easy result. 
Strictly speaking, Step \ref{t1} is redundant by more 
sophisticated arguments in Step \ref{t5} and Step \ref{t6}. 
 
\begin{say}\label{t1}
Let $P$ be an isolated point of $\Nklt (X, \Delta)$. 
Then $P\not\in \Bs|n(K_X+\Delta)|$, 
where $n$ is a divisible 
positive integer. 
\end{say}

\begin{proof}[Proof of {\em{Step \ref{t1}}}]
Let $\mathcal J(X, \Delta)$ be the multiplier ideal sheaf 
associated to $(X, \Delta)$. 
Then we have $$H^i(X, \mathcal O_X(n(K_X+\Delta))\otimes 
\mathcal J(X, \Delta))=0$$ 
for every $i>0$ 
by the Kawamata--Viehweg--Nadel vanishing theorem 
(cf.~\ref{vani}). 
Therefore, the restriction map 
$$
H^0(X, \mathcal O_X(n(K_X+\Delta)))\to 
H^0(X, \mathcal O_X(n(K_X+\Delta))/\mathcal J(X, \Delta))
$$ 
is surjective. 
By assumption, the evaluation map 
$$
H^0(X, \mathcal O_X(n(K_X+\Delta)))\to\mathcal O_X(n(K_X+\Delta))\otimes 
\mathbb C(P) 
$$ 
at $P$ is surjective. 
This implies that 
$P\not \in \Bs|n(K_X+\Delta)|$. 
\end{proof}

Next, we will check that $\Bs |n(K_X+\Delta)|$ 
contains no non-klt centers for 
a divisible positive 
integer $n$ from 
Step \ref{t3} to Step \ref{t7} (cf.~\cite[Theorem 12.1]{fujino2} 
and \cite[Theorem 1.1]{fujino5}). 

\begin{say}\label{t3}
We consider $A_j$ with $\Nlc (X, \Delta)\cap A_j\ne \emptyset$. 
Let $A_j=\sum_i D_i$ be the irreducible decomposition. 
We can easily check that $D_i$ is rational for every $i$ and 
that there exists a point $P\in \Nlc (X, \Delta)$ such that 
$P\in D_i$ for every $i$ by calculating differents 
(see, for example, \cite[Section 14]{fujino2}). 
We can also see that $D_k\cap D_l =P$ for $k\ne l$ and 
that $D_i$ is smooth outside $P$ for every $i$ by adjunction and 
inversion of 
adjunction. 
If $D_i\cap (\Delta -D_i)\ne \emptyset$, 
then 
$D_i$ spans a $(K_{X}+D_i)$-negative extremal ray. 
So, we can contract $D_i$ 
in order to prove that $\Bs |n(K_X+\Delta)|$ contains 
no non-klt centers (see Remark \ref{rem-rem} below). 
We note that $(K_X+\Delta)\cdot D_i=0$. 
Therefore, 
by replacing $X$ with its contraction, 
we may assume that $A_j$ is irreducible. 
We can further assume that $A_j$ is isolated in $\Supp \Delta$. 
It is because we can contract 
$A_j$ if $A_j$ is not isolated in $\Supp \Delta$.  

If $A_j$ is $\mathbb P^1$, 
then it is easy to see that 
$\mathcal O_{A_j}(n(K_X+\Delta))\simeq \mathcal O_{A_j}$ 
since $A_j\cdot (K_X+\Delta)=0$. 

If $A_j\ne \mathbb P^1$, then we obtain $H^1(A_j, \mathcal O_{A_j})\ne 0$. 
Therefore, by Serre duality, we obtain 
$H^0(A_j, \omega_{A_j})\ne 0$, 
where $\omega_{A_j}$ is the dualizing sheaf of $A_j$. 
We note that 
$$
0\to \mathcal T\to \mathcal O_X(K_X+A_j)\otimes \mathcal O_{A_j}\to 
\omega_{A_j}\to 0 
$$ is exact, where $\mathcal T$ is the 
torsion part of $\mathcal O_X(K_X+A_j)\otimes 
\mathcal O_{A_j}$. 
See Lemma \ref{adj} below.  
Since $A_j$ is a curve, $\mathcal T$ is a skyscraper 
sheaf on $A_j$. So, $H^0(A_j, \omega_{A_j})\ne 0$ implies 
$$\Hom (\mathcal O_{A_j}, \mathcal O_X(K_X+A_j)\otimes \mathcal O_{A_j})
\simeq H^0(A_j, \mathcal O_X(K_X+A_j)\otimes \mathcal O_{A_j})\ne 0. $$ 
More precisely, we can lift every section in 
$H^0(A_j, \omega_{A_j})$ to 
$$H^0(A_j, \mathcal O_X(K_X+A_j)\otimes \mathcal O_{A_j})$$ by 
$H^1(A_j, \mathcal T)=0$. 
Therefore, we obtain an inclusion map 
$$
\mathcal O_{A_j}\to \mathcal O_X(n(K_X+A_j))\otimes \mathcal O_{A_j}\simeq 
\mathcal O_{A_j}(n(K_X+\Delta))
$$ for 
a divisible positive integer $n$. 
Since $A_j\cdot (K_X+\Delta)=0$, we see that 
$\mathcal O_{A_j}(n(K_X+\Delta))\simeq 
\mathcal O_{A_j}$. 
\end{say}

The following example may help us understand the case 
when $A_j\ne \mathbb P^1$ in Step \ref{t3}. 

\begin{ex} We consider $C:=(zy^2=x^3)\subset \mathbb P^2=:X$. 
Then $(X, C)$ is not log canonical at 
$P=(0: 0: 1)$. 
On the other hand, $(K_X+C)|_C=K_C\sim 0$ by adjunction. 
\end{ex}

\begin{rem}\label{rem-rem} 
Let $f:(X, \Delta)\to (X', \Delta')$ 
be a proper birational 
morphism between log surfaces such that 
$K_X+\Delta=f^*(K_{X'}+\Delta')$. 
Let $\mathcal C$ be a non-klt center of the pair 
$(X, \Delta)$. 
Then it is obvious that $f(\mathcal C)$ is a non-klt center of 
the pair $(X', \Delta')$. 
Since $\Bs |n(K_X+\Delta)|=f^{-1}\Bs |n(K_{X'}+\Delta')|$ for 
every divisible positive integer 
$n$, 
$\Bs|n(K_X+\Delta)|$ contains no 
non-klt centers of $(X, \Delta)$ if $\Bs |n(K_{X'}+\Delta')|$ contains no 
non-klt centers of $(X', \Delta')$.  
\end{rem}

\begin{say}\label{t2}
If $\Nlc (X, \Delta)\cap A_j=\emptyset$, then 
$\mathcal O_{A_j}(n(K_X+\Delta))\simeq \mathcal O_{A_j}$ for some 
divisible positive integer $n$ 
by the abundance theorem for semi log canonical 
curves (cf.~\cite{fujino1} and \cite[Theorem 1.3]{fuji-gon}). 
\end{say}
 
Anyway, we obtain $\mathcal O_{A}(n(K_X+\Delta))\simeq \mathcal O_A$ for a 
divisible positive integer $n$. 

\begin{say}\label{t4} 
We have $A\cap \Bs|n(K_X+\Delta)|=\emptyset$. 
\end{say}
\begin{proof}[Proof of {\em{Step \ref{t4}}}]
Let $f:Y\to X$ be a resolution such that 
$K_Y+\Delta_Y=f^*(K_X+\Delta)$. 
We may assume that 
\begin{itemize}
\item[(1)] $f^{-1}(A)$ has simple normal crossing 
support, and 
\item[(2)] $\Supp f^{-1}_*\Delta\cup \Exc (f)$ is a 
simple normal crossing divisor on $Y$. 
\end{itemize}
Let $W_1$ be the union of the irreducible components 
of $\Delta^{= 1}_Y$ which 
are mapped into $A$ by $f$. 
We write $\Delta^{= 1}_Y=W_1+W_2$. 
Then 
$$
-W_1-\llcorner \Delta^{>1}_Y
\lrcorner+\ulcorner -(\Delta^{<1}_Y)\urcorner-
(K_Y+\{\Delta_Y\}+W_2)\sim _{\mathbb Q}-f^*(K_X+\Delta). 
$$ 
We put 
$$\mathcal J_1=f_*\mathcal O_Y(-W_1-\llcorner \Delta^{>1}_Y\lrcorner+
\ulcorner -(\Delta^{<1}_Y)\urcorner)\subset \mathcal O_X. 
$$ 
Then we can easily check that 
$$
0\to \mathcal J_1\to \mathcal O_X(-A)\to \delta\to 0 
$$
is exact, where $\delta$ is a skyscraper sheaf, 
and 
$$H^i(X, \mathcal O_X(n(K_X+\Delta))\otimes \mathcal J_1)=0$$ for 
every $i>0$ by \ref{vani}, where 
$n$ is a divisible positive integer. 
By the above exact sequence, we obtain 
$$H^i(X, \mathcal O_X(n(K_X+\Delta))\otimes \mathcal O_X(-A))=0$$ for 
$i>0$. 
By this vanishing theorem, 
we see that the restriction map  
$$
H^0(X, \mathcal O_X(n(K_X+\Delta)))\to 
H^0(A, \mathcal O_A(n(K_X+\Delta)))
$$ is 
surjective. Since $\mathcal O_A(n(K_X+\Delta))\simeq \mathcal O_A$, 
we have $\Bs |n(K_X+\Delta)|\cap A=\emptyset$. 
\end{proof}

\begin{say}\label{t5}
Let $P$ be a zero-dimensional lc center of $(X, \Delta)$. 
Then $P\not \in \Bs|n(K_X+\Delta)|$, 
where $n$ is a divisible positive integer. 
\end{say}
\begin{proof}[Proof of {\em{Step \ref{t5}}}]
If $P\in A$, then it is obvious by Step \ref{t4}. So, we may assume that 
$P\cap \Supp A=\emptyset$. 
Let $f:Y\to X$ be the resolution as in the proof of Step \ref{t4}. 
We can further assume that 
\begin{itemize}
\item[(3)] $f^{-1}(P)$ has simple normal crossing support. 
\end{itemize}
Let $W_3$ be the union of the irreducible components of $\Delta^{=1}_Y$ which 
are mapped into $A\cup P$ by $f$. 
We put $\Delta^{=1}_Y=W_3+W_4$. 
Then $$
-W_3-\llcorner \Delta^{>1}_Y\lrcorner+\ulcorner -(\Delta^{<1}_Y)\urcorner-
(K_Y+\{\Delta_Y\}+W_4)\sim _{\mathbb Q}-f^*(K_X+\Delta). 
$$ 
We put $$\mathcal J_2=f_*\mathcal O_Y(-W_3-\llcorner \Delta^{>1}_Y
\lrcorner+\ulcorner -(\Delta^{<1}_Y)
\urcorner)\subset \mathcal O_X. $$  
Then, we have 
$$H^i(X, \mathcal O_X(n(K_X+\Delta))\otimes \mathcal J_2)=0$$ for 
every $i>0$ by \ref{vani}, 
where $n$ is a divisible positive integer. 
Thus, the restriction map 
$$
H^0(X, \mathcal O_X(n(K_X+\Delta)))\to H^0(X, 
\mathcal O_X(n(K_X+\Delta))\otimes \mathcal 
O_X/\mathcal J_2)
$$ 
is surjective. 
Therefore, the evaluation map 
$$
H^0(X, \mathcal O_X(n(K_X+\Delta)))\to 
\mathcal O_X(n(K_X+\Delta))\otimes \mathbb C(P)
$$ 
is surjective since $P\cap \Supp A=\emptyset$. 
So, we have $P\not \in \Bs|n(K_X+\Delta)|$. 
\end{proof}

\begin{say}\label{t6}
Let $P\in \Nlc (X, \Delta)$. 
Then $P\not\in \Bs|n(K_X+\Delta)|$. 
\end{say}
\begin{proof}[Proof of {\em{Step \ref{t6}}}]
If $P\in A$, then it is obvious by Step \ref{t4}. 
So, we may assume that 
$P\cap \Supp A=\emptyset$. 
By the proof of Step \ref{t4}, 
we obtain that the restriction map 
$$
H^0(X, \mathcal O_X(n(K_X+\Delta)))\to 
H^0(X, \mathcal O_X(n(K_X+\Delta))\otimes 
\mathcal O_X/\mathcal J_1)
$$ 
is surjective. 
Since $P\cap \Supp A=\emptyset$, we see that 
the evaluation map 
$$
H^0(X, \mathcal O_X(n(K_X+\Delta)))\to 
\mathcal O_X(n(K_X+\Delta))\otimes \mathbb C(P)
$$ 
is surjective. 
So, we have $P\not \in \Bs|n(K_X+\Delta)|$.
\end{proof}

\begin{say}\label{t7}
We see that $E_i\not \subset \Bs|n(K_X+\Delta)|$, 
where $E_i$ is any irreducible component of $B$ and 
$n$ is a divisible positive integer.  
\end{say}
\begin{proof}[Proof of {\em{Step \ref{t7}}}]
We may assume that $E_i\cap A=\emptyset$ by Step \ref{t4} 
and $(X, \Delta)$ is log canonical 
in a neighborhood of $E_i$ by Step \ref{t6}. 
We note that $\mathcal O_{E_i}(n(K_X+\Delta))$ is ample. 
So, $\mathcal O_{E_i}(n(K_X+\Delta))$ is generated by global sections. 
Let $f:Y\to X$ be the resolution as in the proof of Step \ref{t4}. 
We can further assume that 
\begin{itemize}
\item[(4)] $f^{-1}(E_i)$ has simple normal crossing support. 
\end{itemize}
Let $W_5$ be the union of the irreducible 
components of $\Delta^{=1}_Y$ which 
are mapped into $A\coprod E_i$ by $f$. 
We put $\Delta^{=1}_Y=W_5+W_6$. 
Then $$
-W_5-\llcorner \Delta^{>1}_Y\lrcorner+\ulcorner 
-(\Delta^{<1}_Y)\urcorner-
(K_Y+\{\Delta_Y\}+W_6)\sim _{\mathbb Q}-f^*(K_X+\Delta). 
$$ 
We put $$\mathcal J_3=f_*\mathcal O_Y(-W_5-\llcorner \Delta^{>1}_Y
\lrcorner+\ulcorner -(\Delta^{<1}_Y)\urcorner)\subset \mathcal O_X.$$  
Then, we have 
$$H^i(X, \mathcal O_X(n(K_X+\Delta))\otimes \mathcal J_3)=0$$ for 
every $i>0$ by \ref{vani}, 
where $n$ is a divisible positive integer. 
We note that there exists a short exact sequence 
$$
0\to \mathcal J_3\to \mathcal O_X(-A-E_i)\to \delta'\to 0, 
$$ 
where $\delta'$ is a skyscraper sheaf on $X$. 
Thus, 
$$H^i(X, \mathcal O_X(n(K_X+\Delta))\otimes \mathcal O_X(-A-E_i))=0$$ for 
every $i>0$, 
Therefore, the restriction map 
$$
H^0(X, \mathcal O_X(n(K_X+\Delta)))\to 
H^0(E_i, \mathcal O_{E_i}(n(K_X+\Delta)))
$$ 
is surjective since $\Supp E_i\cap \Supp A=\emptyset$. 

This implies that $E_i\not\subset \Bs |n(K_X+\Delta)|$ for 
every irreducible component $E_i$ of $B$. 
\end{proof}

Therefore, we have checked that 
$\Bs|n(K_X+\Delta)|$ contains no non-klt centers of $(X, \Delta)$. 

Finally, we will prove that 
$K_X+\Delta$ is semi-ample. 

\begin{say} 
If $|n(K_X+\Delta)|$ is free, then 
there are nothing to prove. 
So, we assume that $\Bs|n(K_X+\Delta)|\ne\emptyset$. 
We take general members $\Xi_1, \Xi_2, \Xi_3 \in |n(K_X+\Delta)|$ and 
put $\Theta =\Xi_1+\Xi_2+\Xi_3$. 
Then $\Theta$ contains no non-klt centers of $(X, \Delta)$ 
and $K_X+\Delta+\Theta$ is not lc at the generic point of 
any irreducible component of $\Bs|n(K_X+\Delta)|$ 
(see, for example, \cite[Lemma 13.2]{fujino2}). 
We put 
$$
c=\max \{ t \in \mathbb R \,|\, K_X+\Delta+t\Theta {\text{ is 
lc outside }} \Nlc (X, \Delta)\}. 
$$ 
Then we can easily check that $c\in \mathbb Q$ and $0<c<1$. 
In this case, 
$$
K_X+\Delta+c\Theta\sim _{\mathbb Q}(1+cn)(K_X+\Delta)
$$ 
and there exists an lc center $\mathcal C$ of 
$(X, \Delta+c\Theta)$ contained 
in $\Bs|(n(K_X+\Delta)|$. 
We take positive integer $l$ and $m$ such that 
$$
l(K_X+\Delta +c\Theta)\sim mn(K_X+\Delta). 
$$ 
Replace $n(K_X+\Delta)$ with $l(K_X+\Delta+c\Theta)$ and 
apply the previous arguments. 
Then, we obtain $\mathcal C\not\subset \Bs|kl(K_X+\Delta +c\Theta)|$ for 
some positive integer $k$. 
Therefore, we have   
$$\Bs |kmn(K_X+\Delta)|\subsetneq \Bs |n(K_X+\Delta)|.  
$$ 
It is because there is an lc center $\mathcal C$ of 
$(X, \Delta +c\Theta)$ such that 
$\mathcal C\subset \Bs|n(K_X+\Delta)|$, and 
$l(K_X+\Delta +c\Theta)\sim mn(K_X+\Delta)$. 
By noetherian induction, we obtain that 
$(K_X+\Delta)$ is semi-ample. 
\end{say}
We finish the proof of Theorem \ref{41}. 
\end{proof}

We used the following lemma in the proof of Theorem \ref{41}. 

\begin{lem}[Adjunction]\label{adj} 
Let $X$ be a normal projective surface and 
let $D$ be a pure one-dimensional 
reduced irreducible 
closed subscheme. 
Then we have the following short exact sequence{\em{:}} 
$$
0\to \mathcal T\to \omega_X(D)\otimes \mathcal O_D
\to \omega_D\to 0, 
$$ 
where $\mathcal T$ is the torsion part of 
$\omega_X(D)\otimes \mathcal O_D$. 
In particular, $\mathcal T$ is a skyscraper sheaf on $D$. 
\end{lem}
\begin{proof}
We consider the following short exact sequence 
$$
0\to \mathcal O_X(-D)\to \mathcal O_X\to \mathcal O_D\to 0. 
$$ 
By tensoring $\omega_X(D)$, where 
$\omega_X(D)=(\omega_X\otimes \mathcal O_X(D))^{**}$, 
we obtain 
$$
\omega_X(D)\otimes \mathcal O_X(-D)\to \omega_X(D)\to 
\omega_X(D)\otimes \mathcal O_D\to 0. 
$$ 
On the other hand, by taking 
$\mathcal Ext^i_{\mathcal O_X}(\underline{\quad}, 
\omega_X)$, we obtain 
$$
0\to \omega_X\to \omega_X(D)\to \omega_D\simeq 
\mathcal Ext^1_{\mathcal O_X}(\mathcal O_D, \omega_X)\to 0. 
$$
Note that $\omega_X(D)\simeq \mathcal Hom_{\mathcal O_X}
(\mathcal O_X(-D), \omega_X)$. 
The natural 
homomorphism 
$$
\alpha :\omega_X(D)\otimes \mathcal O_X(-D)\to 
\omega_X\simeq (\omega_X(D)\otimes \mathcal O_X(-D))^{**}
$$ 
induces the following commutative 
diagram. 
$$
\xymatrix{
& & & 0\ar[d]& \\ 
& & & \mathcal T\ar[d]& \\
 & \omega_X(D)\otimes \mathcal O_X(-D)\ar[d]^{\alpha }\ar[r]&
 \omega_X(D)\ar@{=}[d]\ar[r]&\omega_X(D)\otimes \mathcal O_D \ar[d]
\ar[r] &0\\
0\ar[r] & \omega_X\ar[d]
\ar[r]&\omega_X(D)\ar[r]& \omega_D\ar[r]\ar[d] & 0\\ 
& \mathcal T\ar[d]& & 0& 
\\ & 0&&& }
$$
It is easy to see that 
$\alpha$ is surjective in codimension one and 
$\mathcal T$ is the torsion part of 
$\omega_X(D)\otimes \mathcal O_D$. 
\end{proof}

The next theorem is a generalization of 
Fujita's result in \cite{fujita}. 

\begin{thm}[Finite generation of log canonical rings]\label{43} 
Let $(X, \Delta)$ be a $\mathbb Q$-factorial 
projective log surface such that 
$\Delta$ is a $\mathbb Q$-divisor. 
Then the log canonical 
ring 
$$
R(X, \Delta)=\bigoplus _{m\geq 0}
H^0(X, \mathcal O_X(\llcorner m(K_X+\Delta)\lrcorner))
$$ 
is a finitely generated $\mathbb C$-algebra. 
\end{thm}
\begin{proof}
Without loss of generality, we may assume that 
$\kappa (X, K_X+\Delta)\geq 0$. 
By Theorem \ref{thm32}, 
we may further assume that $K_X+\Delta$ is nef. 
If $K_X+\Delta$ is big, then $K_X+\Delta$ is semi-ample by 
Theorem \ref{41}. 
Therefore, $R(X, \Delta)$ is finitely generated. 
If $\kappa (X, K_X+\Delta)=1$, 
then we can easily check that 
$\kappa (X, K_X+\Delta)=\nu(X, K_X+\Delta)=1$ and 
that 
$K_X+\Delta$ is semi-ample (cf.~\cite[(4.1) Theorem]{fujita}). 
So, $R(X, \Delta)$ is finitely generated. 
If $\kappa (X, K_X+\Delta)=0$, then 
it is obvious that 
$R(X, \Delta)$ is finitely generated. 
\end{proof}

As a corollary, we obtain the finite generation of canonical 
rings for projective surfaces with only rational singularities. 

\begin{cor}\label{44} 
Let $X$ be a projective surface with only rational singularities. 
Then the canonical ring 
$$
R(X)=\bigoplus _{m\geq 0} H^0(X, \mathcal O_X(mK_X))
$$ 
is a finitely generated $\mathbb C$-algebra.
\end{cor}

\begin{rem}
In Theorems \ref{41} and \ref{43}, the assumption that $\Delta$ is a 
boundary $\mathbb Q$-divisor is crucial. 
By Zariski's example, we can easily construct a smooth 
projective surface $X$ and an 
effective $\mathbb Q$-divisor $\Delta$ on $X$ 
such that $\Supp \Delta$ is simple normal crossing, $K_X+\Delta$ 
is nef and big, and 
$$
R(X, \Delta)=\bigoplus _{m\geq 0}H^0(X, \mathcal O_X(\llcorner 
m(K_X+\Delta)\lrcorner))
$$ 
is not a finitely generated $\mathbb C$-algebra. 
Of course, $K_X+\Delta$ is not semi-ample. 
See, for example, \cite[2.3.A Zariski's Construction]{laz}.  
\end{rem}

\section{Non-vanishing theorem}\label{sec5} 

In this section, we prove the following non-vanishing theorem. 

\begin{thm}[Non-vanishing theorem]\label{thm-non} 
Let $(X, \Delta)$ be a $\mathbb Q$-factorial 
projective log surface such that 
$\Delta$ is a $\mathbb Q$-divisor. 
Assume that $K_X+\Delta$ is pseudo-effective. 
Then $\kappa (X, K_X+\Delta)\geq 0$. 
\end{thm}
\begin{proof}
By Theorem \ref{thm32}, we may assume that 
$K_X+\Delta$ is nef. Let $f:Y\to X$ be 
the minimal resolution. We put $K_Y+\Delta_Y=f^*(K_X+\Delta)$. 
We note that $\Delta_Y$ is effective. 
If $\kappa (Y, K_Y)\geq 0$, then it is obvious 
that $$
\kappa (X, K_X+\Delta)=\kappa (Y, K_Y+\Delta_Y)\geq \kappa (Y, K_Y)\geq 0. 
$$ 
So, from now on, we assume $\kappa (Y, K_Y)=-\infty$. 
When $Y$ is rational, we can easily check 
$\kappa (Y, K_Y+\Delta_Y)\geq 0$ by the Riemann--Roch 
formula (see, for example, the proof of \cite[11.2.1 Lemma]{11}). 
Therefore, we may assume that $Y$ is an irrational  
ruled surface. 
Let $p:Y\to C$ be the Albanese fibration. 
We can write $K_Y+\Delta_Y=K_Y+\Delta _1+\Delta_2$, where 
$\Delta_1$ is an effective $\mathbb Q$-divisor 
on $Y$ such that $\Delta_1$ has no vertical components with respect to 
$p$, 
$0\leq \Delta_1\leq \Delta _Y$, $(K_Y+\Delta_1)\cdot F=0$ for 
any general fiber $F$ of $p$, and $\Delta_2=\Delta_Y-\Delta_1\geq 0$. 
When we prove $\kappa (Y, K_Y+\Delta_Y)\geq 0$, 
we can replace $\Delta_Y$ with $\Delta_1$ because 
$\kappa (Y, K_Y+\Delta_Y)\geq \kappa (Y, K_Y+\Delta_1)$. 
Therefore, we may assume that $\Delta_Y=\Delta_1$.  
By taking blow-ups, we can 
further assume that $\Supp \Delta_Y$ is smooth. 
We note the following easy but important lemma. 
\begin{lem}\label{lem-fac}
Let $B$ be any smooth irreducible curve on $Y$ such that 
$p(B)=C$. 
Then $B$ is not $f$-exceptional. 
\end{lem}
\begin{proof}[{Proof of {\em{Lemma \ref{lem-fac}}}}]
Let $\{E_i\}_{i\in I}$ be the set of all $f$-exceptional 
divisors. 
We consider the subgroup $G$ of $\Pic (B)$ generated by 
$\{\mathcal O_B(E_i)\}_{i\in I}$. 
Let $\mathcal L=\mathcal O_C(D)$ be a sufficiently general 
member of $\Pic ^0(C)$. 
We note that the genus $g(C)$ of $C$ is positive. 
Then 
$$(p|_B)^*\mathcal L\in \Pic ^0(B)
\otimes _{\mathbb Z}\mathbb Q\setminus 
G\otimes _{\mathbb Z}\mathbb Q. $$ 
Suppose that $B$ is $f$-exceptional. 
We consider $E=p^*D$ on $Y$. 
Since $X$ is $\mathbb Q$-factorial, 
$$
E\sim _{\mathbb Q}f^*f_*E+\sum _{i\in I}a_i E_i 
$$ with 
$a_i\in \mathbb Q$ for every $i$. 
By restricting the above relation to $B$, we obtain 
$(p|_B)^*\mathcal L\in G\otimes _{\mathbb Z}\mathbb Q$. 
It is a contradiction. Therefore, $B$ is not $f$-exceptional.  
\end{proof}
Thus, every irreducible component $B$ of $\Delta_Y$ 
is not $f$-exceptional. So, its coefficient in $\Delta_Y$ 
is not greater than one because 
$\Delta$ is a boundary $\mathbb Q$-divisor. 
By applying \cite[(2.2) Theorem]{fujita}, 
we obtain that 
$\kappa (Y, K_Y+\Delta_Y)\geq 0$. 
We finish the proof. 
\end{proof}

In \cite{tanaka}, Hiromu Tanaka generalizes Lemma \ref{lem-fac} 
as follows. It is one of the key observations for the 
minimal model theory of log surfaces in positive characteristic.  

\begin{thm}\label{thm53}
Let $k$ be an algebraically closed filed of any characteristic 
such that $k\ne \overline{\mathbb F}_p$. 
We assume that everything is defined over $k$ in this theorem.  
Let $X$ be a $\mathbb Q$-factorial projective 
surface and let $f:Y\to X$ be a projective 
birational 
morphism from a smooth projective surface $Y$. Let $p:Y\to C$ be 
a projective surjective morphism onto a projective smooth 
curve $C$ with the genus $g(C)\geq 1$. 
Then every $f$-exceptional curve $E$ on $Y$ is contained in a fiber of 
$p:Y\to C$. 
\end{thm}

\section{Abundance theorem for log surfaces}\label{sec6} 

In this section, we prove the log 
abundance theorem for $\mathbb Q$-factorial 
projective log surfaces. 

\begin{thm}[Abundance theorem]\label{aban} 
Let $(X, \Delta)$ be a $\mathbb Q$-factorial projective 
log surface such that 
$\Delta$ is a $\mathbb Q$-divisor. 
Assume that $K_X+\Delta$ is nef. 
Then $K_X+\Delta$ is semi-ample. 
\end{thm}

\begin{proof}
By Theorem \ref{thm-non}, we have 
$\kappa (X, K_X+\Delta)\geq 0$. 
If $\kappa (X, K_X+\Delta)=2$, 
then $K_X+\Delta$ is semi-ample by Theorem \ref{41}. 
If $\kappa (X, K_X+\Delta)=1$, then 
$\kappa (X, K_X+\Delta)=\nu(X, K_X+\Delta)=1$ and 
we can easily check that $K_X+\Delta$ is semi-ample 
(cf.~\cite[(4.1) Theorem]{fujita}). 
Therefore, all we have to do is to prove 
$K_X+\Delta\sim _{\mathbb Q}0$ when $\kappa (X, K_X+\Delta)=0$. 
It is Theorem \ref{final} below.  
\end{proof}

The proof of the following theorem depends on the argument in 
\cite[\S 5.~The case $\kappa =0$]{fujita} and 
Sakai's classification result in \cite{sakai1}. 

\begin{thm}\label{final}
Let $(X,\Delta)$ be a $\mathbb Q$-factorial 
projective log surface such that 
$\Delta$ is a $\mathbb Q$-divisor. 
Assume that 
$K_X+\Delta$ is nef and 
$\kappa (X, K_X+\Delta)=0$. 
Then $K_X+\Delta\sim _{\mathbb Q}0$. 
\end{thm}

\begin{proof}
Let $f:V\to X$ be the minimal resolution. 
We put $K_V+\Delta_V=f^*(K_X+\Delta)$. 
We note that $\Delta_V$ is effective. 
It is sufficient to see that $K_V+\Delta_V\sim _{\mathbb Q}0$. 
Let 
$$
\varphi:V=:V_0\overset{\varphi_0}\to V_1\overset{\varphi_1}\to \cdots 
\overset{\varphi_{k-1}}\to V_k=:S 
$$ 
be a sequence of blow-downs such that 
\begin{itemize}
\item[(1)] $\varphi_i$ is a blow-down of a $(-1)$-curve $C_i$ on $V_i$, 
\item[(2)] $\Delta_{V_{i+1}}=\varphi_{i*}\Delta_{V_i}$, and 
\item[(3)] $(K_{V_i}+\Delta_{V_i})\cdot C_i=0$, 
\end{itemize}
for every $i$. 
We may assume that there are no $(-1)$-curves $C$ on $S$ with 
$(K_S+\Delta_S)\cdot C=0$. We note that 
$K_V+\Delta_V=\varphi^*(K_S+\Delta_S)$. 
It is 
sufficient to 
see that $K_S+\Delta_S\sim _{\mathbb Q}0$. 
By assumption, there is a member $Z$ of 
$|m(K_S+\Delta_S)|$ for some divisible positive integer $m$. 
Then, for every positive integer $t$, 
$tZ$ is the unique member of $|tm(K_S+\Delta_S)|$. 
We can easily check the following lemma. 
See, for example, \cite[(5.4)]{fujita}. 
 
\begin{lem}[{cf.~\cite[(5.5) Lemma]{fujita}}]\label{633}  
Let $Z=\sum _i \xi_i Z_i$ be the 
prime decomposition of $Z$. Then 
$K_S\cdot Z_i=\Delta_S \cdot Z_i=Z\cdot Z_i=0$ for every $i$. 
\end{lem} 
We will derive a contradiction assuming $Z\ne 0$, 
equivalently, $\nu(S, K_S+\Delta_S)=1$. 
We can decompose $Z$ into the connected components as follows: 
$$
Z=\sum _{i=1}^{r} \mu_i Y_i, 
$$ 
where $\mu_i Y_i$ is a connected component of 
$Z$ such that $\mu_i$ is the greatest common divisor of 
the coefficients of prime components of $Y_i$ in $Z$ for every 
$i$, 
and $\mu_iY_i\ne \mu_j Y_j$ for $i\ne j$. 
Then we obtain $\omega _{Y_i}\simeq \mathcal O_{Y_i}$ for 
every $i$. 
It is because $Y_i$ is {\em{indecomposable of 
canonical type}} in the sense of Mumford by Lemma \ref{633} 
 (see, for example, \cite[(5.6)]{fujita}). 
\setcounter{say}{0}
\begin{say}[{cf.~\cite[(5.7)]{fujita}}]\label{ste1} 
We assume that $\kappa (S, K_S)\geq 0$. 
Since $0\leq \kappa (S, K_S)\leq \kappa (S, K_S+\Delta_S)=0$, we 
obtain $\kappa (S, K_S)=0$. If 
$S$ is not minimal, then we can find a 
$(-1)$-curve 
$E$ on $S$ such that $E\cdot (K_S+\Delta_S)=0$. 
Therefore, $S$ is minimal by the construction of $(S, \Delta_S)$. 
We show $\kappa (S, K_S+\Delta_S)=\kappa (S, Z)\geq 1$ 
in order to get a contradiction. 
By taking an {}\'etale cover, we may assume that 
$S$ is an Abelian surface or a $K3$ surface. 
In this case, it is easy to see that 
$\kappa (S, K_S+\Delta_S)=\kappa (S, Z)\geq 1$ since 
$Z\ne 0$. 
\end{say}
From now on, we assume that 
$\kappa (S, K_S)=-\infty$. 

\begin{say}
We further assume that $H^1(S, \mathcal O_S)=0$. 
If $n(S, K_S+\Delta_S)=1$, then 
there exist a surjective morphism $g:S\to T$ onto a smooth projective 
curve $T$ and a nef $\mathbb Q$-divisor $A\not\equiv 0$ on $T$ such that 
$K_S+\Delta_S\equiv g^*A$ (cf.~\cite[Proposition 2.11]{8a}). 
Here, $g$ is the reduction map associated to $K_S+\Delta_S$. 
Since $H^1(S, \mathcal O_S)=0$, we 
obtain $K_S+\Delta_S\sim _{\mathbb Q}g^*A$. 
Therefore, $\kappa (S, K_S+\Delta_S)=1$ because $A$ is an ample 
$\mathbb Q$-divisor on $T$. 
It is a contradiction. 
\end{say}

\begin{say}
Under the assumption that 
$H^1(S, \mathcal O_S)=0$, 
we further assume that $n(S, K_S+\Delta_S)=2$. 
By \cite[Proposition 4]{sakai1}, we know $r=1$, that is, 
$Z=\mu_1Y_1$. 
In this case, $S$ is a degenerate del Pezzo surface, 
that is, nine times blow-ups of 
$\mathbb P^2$, and $Z\in |-nK_S|$ for some 
positive integer $n$ (cf.~\cite[Proposition 5]{sakai1}). 
Since $\kappa (S, -K_S)=0$ and 
$m(K_S+\Delta _S)\sim Z\sim -nK_S$, we obtain $m\Delta_S=(m+n)D$, 
where $D$ is the unique member of $|-K_S|$. 
Thus, 
\begin{align*}
\Delta_S=\frac{m+n}{m}D \quad {\text{and}}\quad  
Z=nD. 
\end{align*} 
In particular, we obtain $\Delta_S=\Delta^{>1}_S$. 
We will see that 
$\mathcal O_D(aD)\simeq \mathcal O_D$ for some positive 
integer $a$ in Step \ref{s4}. 
This implies that 
the normal bundle $\mathcal N_D=\mathcal O_D(D)$ is 
a torsion. 
It is a contradiction by \cite[Proposition 5]{sakai1}. 
\end{say}

\begin{say}\label{s4}
In this step, we will prove that 
$\mathcal O_D(aD)\simeq \mathcal O_D$ for 
some positive integer $a$. 
We put $D_k=D$ and construct $D_i$ inductively. 
It is easy to see that 
$\varphi_i:V_i\to V_{i+1}$ is the blow-up at $P_{i+1}$ with 
$\mult _{P_{i+1}}\Delta_{V_{i+1}}\geq 1$ 
for every $i$ by calculating discrepancy 
coefficients since $\Delta_{V_i}$ is effective. 
If $\mult _{P_{i+1}}D_{i+1}=0$, 
then we put $D_i=\varphi^*_{i+1}D_{i+1}$. 
If $\mult _{P_{i+1}}D_{i+1}>0$, 
then we put $D_i=\varphi^*_{i+1}D_{i+1}-C_i$, where 
$C_i$ is the exceptional curve of $\varphi_i$. 
We note that $\mult _P\Delta_{V_{i+1}}>\mult _PD_{i+1}$ for 
every $P\in V_{i+1}$ and $\mult _PD_{i+1}\in \mathbb Z$. 
Finally, we obtain $D_0$ on $V_0=V$. 
We can see that $D_0$ is effective and 
$\Supp D_0\subset \Supp \Delta^{>1}_V$ 
by the above construction. 
We note that 
$\varphi_{i*}\mathcal O_{D_i}\simeq \mathcal O_{D_{i+1}}$ 
for every $i$. 
It is because $\varphi_{i*}\mathcal O_{V_i}(-D_i)\simeq 
\mathcal O_{V_{i+1}}(-D_{i+1})$ and 
$R^1\varphi_{i*}\mathcal O_{V_{i}}(-D_{i})=0$ for every $i$. 
See the following commutative diagram. 
$$
\xymatrix{
0\ar[r] & \mathcal O_{V_{i+1}}(-D_{i+1})
\ar[d]^{\simeq }\ar[r]&\mathcal O_{V_{i+1}}
\ar[d]^{\simeq }\ar[r]&\mathcal O_{D_{i+1}} \ar[d]
\ar[r] &0\\
0\ar[r] & \varphi_{i*}\mathcal O_{V_i}(-D_i)
\ar[r]&\varphi_{i*}\mathcal O_{V_{i}}\ar[r]& 
\varphi_{i*}\mathcal O_{D_{i}}\ar[r] & 
R^1\varphi_{i*}\mathcal O_{V_i}(-D_i)=0 
}
$$ 
Therefore, we obtain $\varphi_*\mathcal O_{D_0}\simeq \mathcal O_D$. 
Since $\Supp D_0\subset \Supp \Delta^{>1}_V$, we see that 
$D_0$ is $f$-exceptional. 
Since $K_V+\Delta_V=f^*(K_X+\Delta)$, 
we obtain 
$\mathcal O_{D_0}(b(K_V+\Delta_V))\simeq 
\mathcal O_{D_0}$ for 
some positive divisible integer $b$. 
Thus, $$\mathcal O_D(b(K_S+\Delta_S))
\simeq\varphi_*\mathcal O_{D_0}
(b(K_V+\Delta_V))\simeq \mathcal O_D. $$ 
In particular, $\mathcal O_D(aD)\simeq \mathcal O_D$ for 
some positive integer $a$. 
It is because 
$$
b(K_S+\Delta_S)\sim \frac{bn}{m}D. 
$$
\end{say}

\begin{say}\label{ste5}
Finally, we assume that 
$S$ is an irrational ruled surface. 
Let $\alpha :S\to B$ be the Albanese fibration. 
In this case, we can easily check that 
every irreducible component of 
$\Supp \Delta^{>1}_S$ is vertical with respect to 
$\alpha$ (cf.~Lemma \ref{lem-fac}). 
Therefore, \cite[(5.9)]{fujita} works without any changes. 
Thus, we get a contradiction. 
\end{say}
We finish the proof of Theorem \ref{final}. 
\end{proof}

\begin{rem}\label{rem-ta} 
In \cite{tanaka}, Hiromu Tanaka slightly 
simplifies the proof of Theorem \ref{final}. 
His proof, which does not use the reduction map (cf.~\ref{2525}), 
works over any algebraically closed field $k$ with $k\ne \overline {\mathbb F}_p$. 
\end{rem}

\begin{rem}\label{rem65} 
Our proof of Theorem \ref{final} 
works over any algebraically closed field $k$ of characteristic zero if we 
use Theorem \ref{thm53} in Step \ref{ste5}. 
From Step \ref{ste1} to Step \ref{s4}, 
we can use the Lefschetz principle because we do not need the $\mathbb Q$-factoriality 
of $X$ there. 
\end{rem}

We close this section with the following corollary. 

\begin{cor}[Abundance theorem for log canonical surfaces] 
Let $(X, \Delta)$ be a complete log canonical surface such that 
$\Delta$ is a $\mathbb Q$-divisor. 
Assume that $K_X+\Delta$ is nef. 
Then $K_X+\Delta$ is semi-ample. 
\end{cor}

\begin{proof}
Let $f:V\to X$ be the minimal resolution. 
We put $K_V+\Delta_V=f^*(K_X+\Delta)$. 
Since $(X, \Delta)$ is log canonical, 
$\Delta_V$ is a boundary $\mathbb Q$-divisor. 
Since $V$ is smooth, 
$V$ is automatically projective. 
Apply Theorem \ref{aban} to the pair $(V, \Delta_V)$. 
We obtain $K_V+\Delta_V$ is semi-ample. 
It implies that 
$K_X+\Delta$ is semi-ample. 
\end{proof}

\section{Relative setting}\label{sec7} 

In this section, we discuss the finite generation of log canonical 
rings and the 
log abundance theorem in the relative setting. 

\begin{thm}[Relative finite generation]\label{rel1} 
Let $(X, \Delta)$ be a log surface 
such that $\Delta$ is a $\mathbb Q$-divisor. 
Let $\pi:X\to S$ be a proper surjective morphism onto a variety $S$. 
Assume that $X$ is $\mathbb Q$-factorial or that $(X, \Delta)$ is log canonical. 
Then 
$$
R(X/S, \Delta)=\bigoplus _{m\geq 0} \pi_*\mathcal O_X(\llcorner 
m(K_X+\Delta)\lrcorner)
$$ 
is a finitely generated $\mathcal O_S$-algebra. 
\end{thm}
\begin{proof}(cf.~Proof of Theorem 1.1 in \cite{fujino-finite}). 
When $(X, \Delta)$ is log canonical, 
we replace $X$ with its minimal resolution. 
So, we may always assume that $X$ is $\mathbb Q$-factorial. 
If $\kappa (X_\eta, K_{X_\eta}+\Delta_\eta)=-\infty$, where 
$\eta$ is the generic point of $S$, 
$X_\eta$ is the generic fiber of $\pi$, and 
$\Delta_\eta=\Delta|_{X_\eta}$, 
then the statement is trivial. 
So, we assume that 
$\kappa (X_\eta, K_{X_\eta}+\Delta_\eta)\geq 0$. 
We further assume that $S$ is affine by shrinking $\pi:X\to S$. 
By compactifying $\pi:X\to S$, we may assume that 
$S$ is projective. 
Since $X$ is $\mathbb Q$-factorial, $X$ is automatically 
projective (cf.~Lemma \ref{222}). 
In particular, $\pi$ is projective. 
Let $H$ be a very ample divisor on $S$ and $G$ a general member of 
$|4H|$. We run the log minimal model program 
for $(X, \Delta +\pi^*G)$. 
By Proposition \ref{2727}, 
this log minimal model program is a log minimal model 
program over $S$. 
It is because any $(K_X+\Delta+\pi^*G)$-negative 
extremal ray of $\overline {NE}(X)$ is a 
$(K_X+\Delta)$-negative extremal ray of $\overline {NE}(X/S)$. 
When we prove this theorem, by Theorem \ref{thm32}, 
we may assume that 
$K_X+\Delta+\pi^*G$ is nef over $S$, equivalently, 
$K_X+\Delta+\pi^*G$ is nef. 
By Theorem \ref{aban}, $K_X+\Delta+\pi^*G$ is semi-ample. 
In particular, $K_X+\Delta$ is $\pi$-semi-ample. 
Thus, $$R(X/S, \Delta)=\bigoplus _{m\geq 0}
\pi_*\mathcal O_X(\llcorner 
m(K_X+\Delta)\lrcorner)$$ 
is a finitely generated $\mathcal O_S$-algebra.  
\end{proof}
 
\begin{thm}[Relative abundance theorem]\label{rel2} 
Let $(X, \Delta)$ be a 
log surface such that $\Delta$ is a $\mathbb Q$-divisor. 
Let $\pi:X\to S$ be a proper 
surjective morphism onto a variety $S$. 
Assume that $X$ is $\mathbb Q$-factorial or that 
$(X, \Delta)$ is log canonical. 
We further assume that $K_X+\Delta$ is $\pi$-nef. 
Then $K_X+\Delta$ is $\pi$-semi-ample. 
\end{thm} 
\begin{proof} 
As in the proof of Theorem \ref{rel1}, 
we may always assume that $X$ is $\mathbb Q$-factorial. 
By Theorem \ref{aban}, 
we may assume that $\dim S\geq 1$. 
By Theorem \ref{rel1}, 
we have that 
$$R(X/S, \Delta)=\bigoplus _{m\geq 0}\pi_*\mathcal O_X(\llcorner 
m(K_X+\Delta)\lrcorner)$$ 
is a finitely generated $\mathcal O_S$-algebra. 
It is easy to see that 
$K_{X_\eta}+\Delta_\eta$ is nef and abundant. 
Therefore, $K_X+\Delta$ is $\pi$-semi-ample 
(see, for example, \cite[Lemma 3.12]{fujino-finite}). 
\end{proof}

We recommend the reader to see 
\cite[3.1.~Appendix]{fujino-finite} for related 
topics. 
Here, we give an easy application. 

\begin{thm}\label{atara}  
Let $X$ be a normal algebraic 
variety with only rational singularities and 
let $\pi:X\to S$ be a projective 
morphism onto a variety $S$. 
Assume that $K_X$ is $\pi$-big. 
Then the relative {\em{canonical model}} 
$$
Y=\Proj _{S} \bigoplus _{m\geq 0} \pi_*\mathcal O_X(mK_X)
$$ 
of $X$ over $S$ has only rational singularities. 
\end{thm}
\begin{proof}
By Theorem \ref{thm32} and 
Proposition \ref{2626}, 
we may assume that 
$K_X$ is $\pi$-nef and $\pi$-big. 
By Theorem \ref{rel2}, there 
exists the birational morphism $\varphi:X\to Y$ over $S$ 
induced by the surjection $\pi^*\pi_*\mathcal O_X(lK_X)\to 
\mathcal O_X(lK_X)$ for some positive 
divisible integer $l$. 
We note that $K_X=\varphi^*K_Y$ by construction. 
Let $f:V\to X$ be a resolution 
such that $K_V+\Delta_V=f^*K_X$ 
and that $\Supp \Delta_V$ is a simple 
normal crossing divisor. 
We consider the following 
short exact sequence 
$$
0\to \mathcal J(X, 0)\to \mathcal O_X\to 
\mathcal O_X/\mathcal J(X, 0)\to 
0. 
$$ 
Note that 
$$
-\llcorner 
\Delta_V\lrcorner -(K_V+\{\Delta _V\})\sim _{\mathbb Q}
-f^*K_X\sim _{\mathbb Q}-f^*\varphi^*K_Y
$$ 
and that $\mathcal J(X, 0)=f_*\mathcal O_V
(-\llcorner \Delta_V\lrcorner)$. 
Then we obtain $$R^i\varphi_*\mathcal J(X, 0)=0$$ 
for every $i>0$ by 
the relative Kawamata--Viehweg--Nadel vanishing theorem. 
Since we have 
$$\dim _{\mathbb C}\Supp (\mathcal O_X/\mathcal J(X, 0))=0, $$ 
we obtain $R^i\varphi_*\mathcal O_X=0$ for 
every $i>0$. 
Therefore, $Y$ has only rational singularities since 
$X$ has only rational singularities. 
It is because $R^ig_*\mathcal O_V\simeq R^i\varphi_*\mathcal 
O_X=0$ for every $i>0$, 
where $g=\varphi\circ f:V\to Y$. 
\end{proof}

\section{Abundance theorem for $\mathbb R$-divisors}\label{new-sec8}

In this section, we generalize the relative log abundance theorem (cf.~Theorem \ref{rel2}) 
for $\mathbb R$-divisors. 

\begin{thm}[Relative abundance theorem for $\mathbb R$-divisors]\label{rdiv}
Let $(X, \Delta)$ be a log surface and let $\pi:X\to S$ be a proper surjective 
morphism onto a variety $S$. 
Assume that 
$X$  is $\mathbb Q$-factorial or that 
$(X, \Delta)$ is log canonical. 
We further assume that $K_X+\Delta$ is $\pi$-nef. 
Then $K_X+\Delta$ is $\pi$-semi-ample.  
\end{thm}

The following proof is essentially due to \cite[Proof of Theorem 2.7]{shokurov-models}. 
 
\begin{proof}
As in the proof of Theorem \ref{rel1}, 
we may always assume that 
$X$ is $\mathbb Q$-factorial. 
We put $F=\Supp \Delta$ and consider the 
real vector space $V=\bigoplus _k \mathbb RF_k$, 
where 
$F=\sum_k F_k$ is the irreducible decomposition. 
We put 
$$
\mathcal P=\{D\in V\, |\, (X, D)\ \text{is a log surface}\}.
$$ 
Then it is obvious that 
$$\mathcal P=\{\sum_k d_k F_k\,|\, 0\leq d_k\leq 1\ \text{for 
every $k$}\}.$$ 
Let 
$\{R_{\lambda}\}_{\lambda\in \Lambda}$ be the set 
of all the extremal rays of $\overline {NE}(X/S)$ spanned by curves. 
We put 
$$
\mathcal N=\{D\in \mathcal P\, |\, (K_X+D)\cdot R_{\lambda}\geq 0\ 
\text{for every $\lambda\in \Lambda$}\}. 
$$ 
Then we can prove that $\mathcal N$ is a rational polytope in $\mathcal 
P$ by using 
Proposition \ref{2727} (cf.~\cite[6.2.~First Main Theorem]{shokurov-models}). 
For the proof, see, for example, the proof of 
\cite[Proposition 3.2]{birkar}. 
We note that 
we can easily see 
$$
\mathcal N=\{D\in \mathcal P\, |\, K_X+D\ \text{is nef}\}. 
$$ 
By the above construction, 
$\Delta \in \mathcal N$. 
Let $\mathcal F$ be the minimal face of $\mathcal N$ containing 
$\Delta$. 
Then we can take $\mathbb Q$-divisors $\Delta _1, \cdots, \Delta_l$ on $X$ and 
positive 
real numbers $r_1, \cdots, r_l$ such that 
$\Delta_i$ is in the relative  interior of $\mathcal F$ for 
every $i$, 
$K_X+\Delta=\sum _i r_i(K_X+\Delta_i)$, 
and $\sum _i r_i =1$. 
By Theorem \ref{rel2}, $K_X+\Delta_i$ is $\pi$-semi-ample 
for every $i$ since $K_X+\Delta_i$ is $\pi$-nef. 
Therefore, $K_X+\Delta$ is $\pi$-semi-ample. 
\end{proof}

We note the following easy but important 
remark on Theorem \ref{rdiv}. 

\begin{rem}[Stability of Iitaka fibrations]
In the proof of Theorem \ref{rdiv}, 
we note the following property. 
If $C$ is a curve on $X$ such that 
$\pi(C)$ is a point and $(K_X+\Delta_{i_0})\cdot C=0$ for some $i_0$, 
then $(K_X+\Delta_i)\cdot C=0$ for every $i$. 
It is because we can find $\Delta'\in \mathcal F$ such that 
$(K_X+\Delta')\cdot C<0$ if $(K_X+\Delta_i)\cdot C >0$ for 
some $i\ne i_0$. 
It is a contradiction. 
Therefore, there exist 
a contraction morphism $f:X\to Y$ over $S$ and 
$g$-ample $\mathbb Q$-Cartier 
$\mathbb Q$-divisors 
$A_1, \cdots, A_l$ on $Y$, 
where $g:Y\to S$, such that 
$K_X+\Delta_i\sim _{\mathbb Q}f^*A_i$ for 
every $i$. 
In particular, we obtain 
$$
K_X+\Delta\sim _{\mathbb R}f^*(\sum _i r_i A_i). 
$$ 
Note that $\sum_i r_i A_i$ is $g$-ample. 
Roughly speaking, 
the Iitaka fibration of $K_X+\Delta$ is the same as that 
of $K_X+\Delta_i$ for every $i$. 
\end{rem}

Anyway, we obtain the relative 
log minimal model program for log surfaces 
(cf.~Theorem \ref{thm32}) and 
the relative log abundance theorem for log 
surfaces (cf.~Theorem \ref{rdiv}) 
in full generality. 
Therefore, we can freely use the log minimal model theory for 
log surfaces in the relative setting. 

We close this section with an easy application of 
Theorem \ref{rdiv}. 

\begin{thm}[Base point free theorem via abundance]
Let $(X, \Delta)$ be a log surface and let $\pi:X\to S$ be 
a proper surjective morphism onto a variety $S$. Assume that 
$X$ is $\mathbb Q$-factorial or that 
$(X, \Delta)$ is log canonical. Let $D$ be a $\pi$-nef 
$\mathbb R$-Cartier $\mathbb R$-divisor 
on $X$. If $D-(K_X+\Delta)$ is $\pi$-semi-ample, 
then $D$ is $\pi$-semi-ample. 
\end{thm}

\begin{proof}
If $(X, \Delta)$ is log canonical, then we replace $X$ with its 
minimal resolution. So we may always assume that 
$X$ is $\mathbb Q$-factorial. 
Without loss of generality, we can further assume that 
$S$ is affine. Since $D-(K_X+\Delta)$ is $\pi$-semi-ample, 
we can write 
$$
D-(K_X+\Delta)\sim _{\mathbb R}\Delta'\geq 0
$$
such that $\Delta+\Delta'$ is a boundary $\mathbb R$-divisor on $X$. 
Therefore, we obtain $D\sim _{\mathbb R}K_X+\Delta+\Delta'$. 
By Theorem \ref{rdiv}, we obtain that 
$D$ is $\pi$-semi-ample. 
\end{proof}

\section{Appendix:~Base point free theorem for log surfaces}\label{sec8} 

In this appendix, we prove the base point free theorem for log surfaces 
in full generality. 
It generalizes Fukuda's base point free theorem 
for log canonical surfaces (cf.~\cite[Main Theorem]{fukuda}). 
Our proof is different from Fukuda's and depends on the 
theory of {\em{quasi-log varieties}}. 
We note that this result is not necessary for the minimal model theory 
for log surfaces discussed in this paper. We also note that 
a more general result was stated in \cite[Theorem 7.2]{ambro} 
without any proofs (cf.~\cite[Theorem 4.1]{fujino3}). 

\begin{thm}[Base point free theorem for log surfaces]\label{bpf}
Let $(X, \Delta)$ be a log surface and 
let $\pi:X\to S$ be a proper surjective morphism onto a variety $S$. 
Let $L$ be a $\pi$-nef Cartier 
divisor on $X$. 
Assume that 
$aL-(K_X+\Delta)$ is $\pi$-nef and $\pi$-big and 
that $(aL-(K_X+\Delta))|_C$ is $\pi$-big for every lc 
center $C$ of the pair $(X, \Delta)$, 
where $a$ is a positive number. 
Then there exists a positive integer $m_0$ such that 
$\mathcal O_X(mL)$ is $\pi$-generated 
for every $m\geq m_0$.  
\end{thm}

\begin{rem}
In Theorem \ref{bpf}, 
the condition 
that $(aL-(K_X+\Delta))|_C$ is $\pi$-big 
for every lc center $C$ of the pair $(X, \Delta)$ is equivalent to 
the following condition:~$(aL-(K_X+\Delta))\cdot C>0$ for every 
irreducible component $C$ of $\llcorner \Delta\lrcorner$ such that 
$\pi(C)$ is a point. 
\end{rem}

\begin{proof}
Without loss of generality, we may assume that $S$ is affine since 
the problem is local. 
We divide the proof into several steps. 
\setcounter{say}{0}
\begin{say}[Quasi-log structures] 
Since $(X, \Delta)$ is a log surface, the 
pair $[X, \omega]$, where $\omega=K_X+\Delta$, has a natural 
quasi-log structure. 
It induces a quasi-log structure $[V, \omega']$ on $V=\Nklt (X, \Delta)$ with 
$\omega'=\omega|_V$. 
More precisely, 
let $f:Y\to X$ be a resolution such that 
$K_Y+\Delta_Y=f^*(K_X+\Delta)$ and that 
$\Supp \Delta_Y$ is a simple normal crossing 
divisor on $Y$. 
By the relative Kawamata--Viehweg vanishing theorem, we 
obtain the following short exact sequence 
\begin{align*}
0 &\to f_*\mathcal O_Y(-\llcorner \Delta_Y\lrcorner)\to f_*\mathcal O_Y 
(\ulcorner (-\Delta^{<1}_Y)\urcorner -\llcorner \Delta^{>1}_Y\lrcorner)\\ &\to 
f_*\mathcal O_{\Delta^{=1}_Y}(\ulcorner 
(-\Delta^{<1}_Y)\urcorner-\llcorner \Delta^{>1}_Y\lrcorner)\to 0.  
\end{align*} 
Note that $$-\llcorner\Delta_Y\lrcorner =\ulcorner 
(-\Delta^{<1}_Y)\urcorner -\llcorner \Delta^{>1}_Y\lrcorner -\Delta^{=1}_Y.$$ 
We also note that the scheme structure of $V$ is defined by the 
multiplier ideal sheaf $\mathcal J(X, \Delta)=
f_*\mathcal O_Y(-\llcorner \Delta_Y\lrcorner)$ of the pair $(X, \Delta)$ 
and that $X_{-\infty}$ (resp.~$V_{-\infty}$) is defined by 
the ideal sheaf $f_*\mathcal O_Y 
(\ulcorner (-\Delta^{<1}_Y)\urcorner -\llcorner \Delta^{>1}_Y\lrcorner)=:\mathcal 
I_{X_{-\infty}}$ 
(resp.~$f_*\mathcal O_{\Delta^{=1}_Y}(\ulcorner 
(-\Delta^{<1}_Y)\urcorner-\llcorner \Delta^{>1}_Y
\lrcorner)=:\mathcal I_{V_{-\infty}}$). 
By construction, $X_{-\infty}\simeq V_{-\infty}$ and 
$X_{-\infty}=\Nlc (X, \Delta)$. 
We note the following commutative diagram. 
$$
\xymatrix{
0\ar[r]& \mathcal J(X, \Delta)\ar[r]\ar@{=}[d]&\mathcal I_{X_{-\infty}}\ar[r]
\ar[d]&\mathcal 
I_{V_{-\infty}}\ar[r]\ar[d]& 0\\ 
0\ar[r]& \mathcal J(X, \Delta)\ar[r]&\mathcal O_X\ar[r]&\mathcal 
O_V\ar[r]& 0
}
$$ 
For details, see \cite[Section 4]{ambro}, \cite[Section 3.2]{fujino3}, 
and \cite{fujino-s}. 
\end{say}
\begin{say}[Freeness on $\Nklt (X, \Delta)$]
By assumption, 
$aL|_V-\omega'$ is $\pi$-ample and 
$\mathcal O_{V_{-\infty}}(mL)$ is $\pi|_{V_{-\infty}}$-generated for every 
$m\geq 0$. 
We note that $\dim V\leq 1$ and 
$\dim V_{-\infty}\leq 0$. 
Therefore, by \cite[Theorem 3.66]{fujino3}, 
$\mathcal O_V(mL)$ is $\pi$-generated for every $m\gg 0$.  
\end{say}
\begin{say}[Lifting of sections]
We consider the following short exact sequence 
$$
0\to \mathcal J(X, \Delta)\to \mathcal O_X\to \mathcal O_V\to 0, 
$$ 
where $\mathcal J(X, \Delta)$ is the multiplier ideal sheaf of $(X, \Delta)$. 
Then we obtain that 
the restriction map 
$$
H^0(X, \mathcal O_X(mL))\to H^0(V, \mathcal O_V(mL))
$$ 
is surjective for every $m\geq a$ since 
$$H^1(X, \mathcal J(X, \Delta)\otimes \mathcal O_X(mL))=0$$ for 
$m\geq a$ by the relative Kawamata--Viehweg--Nadel vanishing 
theorem. 
Thus, there exists a positive integer 
$m_1$ such that 
$\Bs |mL|\cap \Nklt (X, \Delta)=\emptyset$ for 
every $m\geq m_1$. 
\end{say}
So, all we have to do is to prove 
that $|mL|$ is free for every $m \gg 0$ 
under the assumption that 
$\Bs|nL|\cap \Nklt (X, \Delta)=\emptyset$ for 
every $n\geq m_1$.
\begin{say}[Kawamata's X-method]\label{step-kawa} 
Let $f:Y\to X$ be a resolution with a simple normal crossing 
divisor $F=\sum _j F_j$ on $Y$. 
We may assume the following conditions. 
\begin{itemize}
\item[(a)] 
$K_Y=f^*(K_X+\Delta)+\sum _j a_j F_j$ for some $a_j \in \mathbb R$. 
\item[(b)] $f^*|p^lL|=|M|+\sum _j r_j F_j$, where 
$|M|$ is free, $p$ is a prime number such that 
$p^l\geq m_1$, and 
$\sum _j r_j F_j$ is the fixed part of 
$f^*|p^lL|$ for some 
$r_j \in \mathbb Z$ with $r_j \geq 0$. 
\item[(c)] $f^*(aL-(K_X+\Delta))-\sum_j \delta_j F_j$ is $\pi$-ample 
for some $\delta_j \in \mathbb R$ with $0<\delta_j \ll 1$. 
\end{itemize}
We set $$
c=\min \left\{\frac{a_j+1-\delta_j}{r_j}\right\}
$$ 
where the minimum is taken for all the $j$ such that 
$r_j\ne 0$. 
Then, we obtain $c>0$. Here, we used the 
fact that $a_j>-1$ if $r_j>0$. 
It is because $\Bs |p^lL|\cap \Nklt (X, \Delta)=\emptyset$. 
By a suitable choice of the $\delta_j$, we may assume that 
the minimum is attained at exactly one value $j=j_0$. 
We put 
$$
A=\sum _j (-cr_j +a_j-\delta_j)F_j. 
$$ 
We consider 
\begin{align*}
N:=&\  p^{l'}f^*L-K_Y+\sum _j (-cr_j +a_j -\delta_j)F_j \\
=& \ (p^{l'}-cp^l-a)f^*L \quad \quad 
({\text{$\pi$-nef if $p^{l'}\geq cp^l+a$}})
\\ &+c(p^lf^*L-\sum _j r_j F_j) \quad \quad 
({\text{$\pi$-free}})\\
&+f^*(aL-(K_X+\Delta))-\sum _j \delta_j F_j 
\quad\quad ({\text{$\pi$-ample}})
\end{align*} 
for some positive integer $l'>l$. 
Then $N$ is $\pi$-ample if $p^{l'}\geq cp^l+a$. 
By the relative Kawamata--Viehweg vanishing theorem, we have 
$$
H^i(Y, \mathcal O_Y(K_Y+\ulcorner N\urcorner))=0 
$$ for every $i>0$. 
We can write $\ulcorner A\urcorner=B-F-D$, 
where $B$ is an effective $f$-exceptional Cartier divisor, 
$F=F_{j_0}$, $D$ is an effective Cartier divisor such that 
$\Supp D\subset \Supp \sum _{a_j\leq -1}F_j$, and 
$\Supp B$, $\Supp F$, and $\Supp D$ have no common irreducible components 
one another by $\Bs |p^lL|\cap \Nklt (X, \Delta)=\emptyset$. 
We note that $K_Y+\ulcorner N\urcorner=p^{l'}f^*L+\ulcorner A\urcorner$. 
Then the restriction map 
\begin{align*}
&H^0(Y, \mathcal O_Y(p^{l'}f^*L+B))\\ &\to 
H^0(F, \mathcal O_F(p^{l'}f^*L
+B))\oplus H^0(D, \mathcal O_D(p^{l'}f^*L+B))
\end{align*} 
is surjective. 
Here, we used the fact that $\Supp F\cap \Supp D=\emptyset$. 
Thus we obtain that 
$$
H^0(X, \mathcal O_X(p^{l'}L))\simeq H^0(Y, \mathcal O_Y(p^{l'}f^*L+B))\to 
H^0(F, \mathcal O_F(p^{l'}f^*L+B))
$$ 
is surjective. 
We note that $H^0(F, \mathcal O_F(p^{l'}f^*L+B))\ne 0$ for every $l'\gg 0$ 
since $F$ is a smooth curve and 
\begin{align*}
N|_F&=(p^{l'}f^*L-K_Y+B-F-D-\{-A\})|_F\\
&=(p^{l'}f^*L+B)|_F-(K_F+\{-A\}|_F)
\end{align*} 
is $\pi$-ample 
(cf.~Shokurov's non-vanishing theorem). 
Therefore, we have 
$\Bs|p^{l'}L|\subsetneq \Bs |p^lL|$ 
for some $l'\gg 0$ since $f(F)\subset \Bs |p^lL|$. 
By noetherian induction, 
we obtain $\Bs|p^kL|=\emptyset$ for some positive integer $k$. 
\end{say}
Let $q$ be a prime number with $q\ne p$. 
Then we can find $k'>0$ such that $\Bs|q^{k'}L|=\emptyset$ by the same argument 
as in Step \ref{step-kawa}. 
So, we can find a positive integer $m_0$ such that 
$\Bs|mL|=\emptyset$ for every $m\geq m_0$. 
\end{proof}


\end{document}